\documentclass[11pt,reqno]{amsart}
\usepackage{amsmath, latexsym, amsfonts, amssymb,amsthm, amscd,epsfig,enumerate, color}
\advance\hoffset -.75cm

\usepackage{amsfonts}
\usepackage{latexsym}
\usepackage{amsmath}
\usepackage{amssymb}

\newcommand{\R}{\mathbb R}
\newcommand{\Z}{\mathbb Z}
\newcommand{\N}{\mathbb N}
\newcommand{\pr}{\mathbb P}
\newcommand{\E}{\mathbb E}

\newcommand{\diff}{\textrm{d}}

\newcommand{\ident}{{\mathchoice {\rm 1\mskip-4mu l} {\rm 1\mskip-4mu l}
{\rm 1\mskip-4.5mu l} {\rm 1\mskip-5mu l}}}

\newcounter{saveenum}

\newtheorem{teo}{Theorem}[section]
\newtheorem{lem}[teo]{Lemma}
\newtheorem{cor}[teo]{Corollary}
\newtheorem{rem}[teo]{Remark}
\newtheorem{pro}[teo]{Proposition}
\newtheorem{defn}[teo]{Definition}

\newtheorem{exmp}[teo]{Example}

\begin{document}

\title[Rumor processes]
{Rumor processes in random environment on $\mathbb{N}$ \\ and on Galton-Watson trees}

\author[D.~Bertacchi]{Daniela Bertacchi}
\address{D.~Bertacchi,  Universit\`a di Milano--Bicocca
Dipartimento di Matematica e Applicazioni,
Via Cozzi 53, 20125 Milano, Italy
}
\email{daniela.bertacchi\@@unimib.it}

%
\author[F.~Zucca]{Fabio Zucca}
\address{F.~Zucca, Dipartimento di Matematica,
Politecnico di Milano,
Piazza Leonardo da Vinci 32, 20133 Milano, Italy.}
\email{fabio.zucca\@@polimi.it}

\date{}

\begin{abstract}
The aim of this paper is to study rumor processes in random environment. 
In a rumor process a signal starts from the stations of a fixed vertex (the root) and 
travels on a graph from vertex to vertex. We consider two rumor processes.
In the firework process each station, when reached by the signal, transmits it up to a random
distance. In the reverse firework process, on the other hand, stations do not send any signal but they
``listen'' for it up to a random distance.
The first random environment that we
consider is the deterministic 1-dimensional tree $\N$ with a random number of stations on each vertex; in this case
the root is the origin of $\N$.
We give conditions for the survival/extinction on almost every realization of the sequence of
stations.
Later on, we study the processes on Galton-Watson trees with random number of stations on each vertex.
We show that if the probability of survival is positive, then there is survival on almost every 
realization of the infinite tree such that there is at least one station at the root.
We characterize the survival of the process in some cases and 
we give sufficient conditions for survival/extinction.
\end{abstract}

\maketitle



\baselineskip .6 cm



\section{Introduction}
\label{sec:intro}

Rumor processes are models for the propagation of signals or information on a net of stations. There is a large
literature on these processes, which takes into account deterministic and stochastic variants both in space with no
structure or on graphs, complex networks and grids (see e.g.~\cite{cf:LMM, cf:LR, cf:MT, cf:Sud} or 
\cite{cf:JMZ} and the references therein). More variants include competing rumors and rumors spreading
in a moving population (see \cite{cf:KS, cf:KOW, cf:KLLM}).
So far, up to our knowledge, 
no rumor models on random environment have been studied. In this paper we study two kinds of discrete-time rumor processes
(the \textit{firework} and  \textit{reverse firework} processes, introduced in \cite{cf:JMZ}) 
with a random number of stations at each vertex
of a graph $\mathbf{X}:=(X,E(X))$ which can be either  $\N$ or  a Galton-Watson tree (briefly, GW tree).
The use of a Galton-Watson tree (see for instance \cite{cf:PemStac1}) allows to study a random process
on a random and heterogeneous graph structure still retaining some probabilistic homogeneity.
The main question about this model is to understand under which conditions, the signal, starting
from one vertex of the graph, will spread indefinitely with positive probability or die out almost surely
in a finite number of steps.
It is worth noting that rumor processes can describe the behavior of other models such as the frog model (see
\cite{cf:BMZ, cf:LMZ10}): for instance the firework process  
 on $\N$ can describe the local survival of a frog model with immortal particles with left drift (where
the radius of a ``station'' or ``frog'' at $x$ is the distance between $x$ and the rightmost vertex ever reached by the frog).

Let $(\Omega,\mathcal{F}, \pr)$ be a probability space and consider
two families
of random variables $\{N_x\}_{x \in X}$ and $\{R_{x,i}\}_{i \in \N^*, x \in X}$ such that
$\{N_x, R_{x,i}\}_{i \in \N^*, x \in X}$ are independent and $\{R_{x,i}\}_{i \in \N^*}$ are identically distributed
for all $x \in X$, that is $R_{x,i}\sim R_x$ (where $\N^*:=\N \setminus\{0\}$). 
We choose an origin  $o \in X$ (which is
the origin of $\N$ or the root of the GW-tree) and we suppose that $\pr(N_o>0)>0$.
If $N_x \sim N$ and $R_{x} \sim R$ for all $x \in X$ (resp.~for all $x \in X \setminus\{o\}$) 
then the firework (resp.~reverse firework) process is called \textit{homogeneous}; 
\textit{heterogeneous} otherwise.
$N_x$ represents the random number of radio stations at $x$ and 
$\{R_{x,i}\}_{i=1}^{N_x}$ are the operating radii of the stations. In order to avoid trivial cases, we assume that $\pr(R_x<1)\in(0,1)$.
By $\mathcal{N}=\{N_x\}_{x \in X}$ we denote the random sequence of the numbers of stations.

We describe now the firework process starting from the root $o$ on a tree $(X,E(X))$, 
with deterministic  number of stations $\{n_x\}_{x \in X}$ and radii $\{r_{x,i}\}_{i \in \N^*, x \in X}$. 
If either the number of stations
or the radii are random (or the tree itself is a random graph), 
then this evolution applies to every fixed configuration $\omega \in \Omega$.
At time $0$ only the stations at $o$, if any, are active. At time $1$ all the stations at $x$ such that 
the distance between $o$ and $x$ is less than or equal to $\max_{i=1, \ldots, n_o} r_{o,i}$ are activated.
Given the set $A_n$ of active stations at time $n$ then 
$x \in A_{n+1}$ if and only if there exists $y \in  A_n$ such that 
\begin{enumerate}[(a)]
 \item 
$x$ belongs to the subtree branching
from $y$
\item
the distance between $y$ and $x$ is less than or equal to $\max_{i=1, \ldots, n_y} r_{y,i}$.
\end{enumerate}
Clearly if $n_o=0$ the process does not start at all.
We say that the process survives if and only if the stations on an infinite number of vertices are activated.
When the number of stations and the radii are random variables and the tree $\mathbf{X}$ is a random 
graph, we say that there is survival 
if the event
\[
 V=\{\omega \in \Omega\colon \textrm{ the firework process on } \mathbf{X}(\omega) \textrm{ with radii }
\{R_{x,i}(\omega)\} \textrm{ and } \{N_x(\omega)\}\textrm{ stations survives}\}
\]
has positive probability. 
We call this \textit{annealed survival} (see Section~\ref{subsec:annealed} 
for the definition of the annealed counterpart of the process).
We are mainly interested in the \textit{quenched survival}, that is, $\pr(V|(\mathbf{X},\mathcal{N})=(\mathbf{T},\mathbf{n}))>0$ for
almost every tree $\mathbf{T}$ and almost every sequence $\mathbf{n}$ of stations in a fixed set 
(see Section~\ref{sec:prelim} for details).

The reverse firework process starting from the root $o$ on a tree $(X,E(X))$, 
with $\{n_x\}_{x \in X}$ stations and radii $\{r_{x,i}\}_{i \in \N^*, x \in X}$ evolves in a slightly
different way. We consider $o$ as an active vertex at time $0$ disregarding its number of stations.
The number of stations at $o$ is not important because, while in the firework process the stations are actively
sending and passively listening, in the reverse firework process it is the other way around. More precisely, 
at time $1$ all the stations at $x$ such that 
the distance between $o$ and $x$ is less than or equal to $\max_{i=1, \ldots, n_x} r_{x,i}$ are activated.
Given the set $A_n$ of active stations at time $n$ then 
$x \in A_{n+1}$ if and only if there exists $y \in  A_n$ such that 
\begin{enumerate}[(a)]
 \item 
$x$ belongs to the subtree branching
from $y$
\item
the distance between $y$ and $x$ is less than or equal to $\max_{i=1, \ldots, n_y} r_{x,i}$.
\end{enumerate}
When the number of stations and the radii are random variables (or the tree is a random graph), 
the annealed and quenched survival for the 
reverse firework process are defined analogously to the firework case.

It has been shown in \cite{cf:JMZ} that, on $\N$ with a deterministic number of stations, even if the two processes
look similar, they behave differently. The same phenomenon can be observed in random environment.

There is a simple way to describe the survival of our processes by using an auxiliary graph.
Consider a tree $(X,E(X))$, a root $o$ and the sequence of \textit{annealed radii}  $\{\widetilde R_x\}_{x \in X}$
where $\widetilde R_x := \ident_{\{N_x \ge 1\}} \cdot \max\{R_{x,j}\colon j=1, \ldots, N_x\}$. 
We associate a new graph, that we call \textsl{F-graph}, with the firework process as follows: the set of vertices is
$X$ and we draw an edge $(w, w^\prime)$ if and only if $w^\prime$ belongs to the subtree
branching from $w$ and the distance between $w$ and $w^\prime$ is less than or equal to $\widetilde R_w$;
the F-graph is the connected component containing $o$.
A similar graph, the \textsl{RF-graph}, can be constructed for the reverse firework process: we draw and edge 
$(w, w^\prime)$ if and only if $w^\prime$ belongs to the subtree
branching from $w$ and the distance between $w$ and $w^\prime$ is less than or equal to $\widetilde R_{w^\prime}$;
as before, we consider just the connected component containing $o$. 
It is clear that there is survival for the firework (resp.~reverse firework)  process
if and only if the  F-graph (resp.~RF-graph)
is infinite.
In Section~\ref{sec:GWtree} the tree and the sequence $\{\widetilde R_w\}_{w \in \mathbf{T}}$ are both random, thus the 
F-graph and RF-graph are random 
as well.

Here is a brief outline of the paper. Section~\ref{sec:prelim} is devoted primarily to the construction
of the probability space, starting with the space of labelled GW-trees, 
where our processes live (Section~\ref{subsec:genericlabelledtrees}). The main result of this section, 
Lemma~\ref{lem:BRWihenrited}, is the key to obtain quenched results from the annealed ones for the processes
on GW-trees.
In Section~\ref{subsec:N} we discuss the important case where the GW-tree is the deterministic
1-dimensional tree $\N$. In Section~\ref{subsec:annealed} we introduce the notion of annealed counterpart
of the firework and reverse firework processes.
The results on survival and extinction for the firework process and the reverse firework process on $\N$ with random number of stations
can be found in Sections~\ref{sec:firework} and \ref{sec:reverse} respectively.
Theorem~\ref{thm:2} gives a characterization for the extinction of the homogeneous firework process. More explicit
conditions for survival/extinction can be found in Proposition~\ref{pro:homogeneous1} and Corollary~\ref{cor:homogeneous1}:
these are conditions on the tails of the distributions $R$ and $N$. In 
Remark~\ref{rem:finiteradiussurvival} we show that, given any law $R$ for the radii (resp.~$N$ for the number of stations)
there exist a law $N$ for the number of stations (resp.~$R$ for the radii) such that the firework process survives.
In particular, unlike the deterministic case, survival is possible even if the expected value of $R$ is finite.
The possible behaviors of the firework process when $\E[R]$ and $\E[N]$ are finite/infinite
are discussed in Remark~\ref{rem:expectedvalues}. Theorem~\ref{thm:3}
gives a sufficient condition for the survival of the heterogeneous 
firework process.
A characterization of survival for the homogeneous reverse 
firework process is given by Theorem~\ref{th:reversehomogeneous} while some sufficient conditions for survival/extinction
(on the tails of the distributions $R$ and $N$)
can be found in Corollary~\ref{cor:GW-Homogeneous-reverse-Firework2}. Remark~\ref{rem:expectedvalues-reverse}
is the reverse firework counterpart of Remark~\ref{rem:finiteradiussurvival}.
Theorem~\ref{th:reverseinhomogeneous}
gives a sufficient condition for the survival of the heterogeneous 
firework process.
Section~\ref{sec:GWtree} is devoted to homogeneous firework and reverse firework processes on GW-trees. 
Lemma~\ref{lem:BRWihenrited2} 
derives quenched results from annealed ones by using Lemma~\ref{lem:BRWihenrited}.
Conditions for survival/extinction can be found in Theorem~\ref{th:GW-Homogeneous-Firework2}
and Corollary~\ref{cor:GW-Homogeneous-Firework} for the firework process and in
Theorem~\ref{th:GW-Homogeneous-reverse-Firework2}
and Corollary~\ref{cor:GW-Homogeneous-reverse-Firework2} for the reverse firework process.  In Example~\ref{exmp:bernoulli}
we apply all these results to the case where, independently, at each vertex of the GW-tree there is a station according
to a Bernoulli distribution. All the proofs, along with some technical lemmas, can be found in Section~\ref{sec:proofs}.

Finally it is worth noting that some results extend by coupling. 
Suppose that we have two families of random variables $\{N_x, R_{x,i}\}_{i \in \N^*, x \in X}$ and 
$\{N^\prime_x, R^\prime_{x,i}\}_{i \in \N^*, x \in X}$
such that $N_x \succeq N^\prime_x$ and $R_{x,i} \succeq R^\prime_{x,i}$ for all $i \in \N^*, x \in X$.
By coupling, if the (firework or reverse firework) process associated with $\{N_x, R_{x,i}\}_{i \in \N^*, x \in X}$ 
dies out almost surely
then so does the process associated with $\{N^\prime_x, R^\prime_{x,i}\}_{i \in \N^*, x \in X}$.

\section{Preliminaries and construction of the process}
\label{sec:prelim}

In this section we construct the space of the processes and we establish the notation that we use in the paper.
This section is organized as follows. We start in Section~\ref{subsec:genericlabelledtrees} 
by constructing a \textit{random labelled GW-tree} or \textit{multiytype GW-tree}
(which generalizes the well-known GW-tree) and the probability space of our processes. 
Section~\ref{subsec:N} is devoted to the special case where the GW-tree is simply $\N$; this is all we need
in Sections~\ref{sec:firework} and \ref{sec:reverse}. In Section~\ref{subsec:annealed}
we introduce the \textit{annealed counterpart} or our processes which is used throughout
the whole paper.

\subsection{Random labelled Galton-Watson trees}
\label{subsec:genericlabelledtrees}

The idea is to construct a probability space for our processes on GW-trees with
random radii and random number of stations. 
To this aim we construct the space of labelled GW-trees which can be seen as the genealogy tree 
associated with a generic discrete-time 
multitype branching process (or MBP). Even though it can be constructed for a generic MBP
(or a generic branching random walk as described in  \cite{cf:BZ, cf:BZ2, cf:BZ3, cf:Z1}), for sake
of simplicity we write the explicit construction for the particular case that we need. Here
the \textit{labels} of the vertices are the number of stations.

The reader who is just interested in
firework and reversed firework processes on a random environment on $\N$ can skip this section and go 
to Section~\ref{subsec:N}.

Let us define the space of unlabelled GW-trees (the usual GW-trees).
Consider a GW-process, with offspring distribution $\rho$. We denote by
$\varphi(z):=\sum_{n \in \N} \rho(n) z^n$ the generating function of the GW-process and we write
$m:=\sum_{n \in \N} n \rho(n)$.
We consider the set of finite words $\mathcal{W}$ (including the empty word, $\emptyset$, which is the root) on the infinite
alphabet $\mathbb{N}^*:=\mathbb{N}\setminus\{0\}$. Given two words $w,w^\prime \in \mathcal{W}$ we denote by $ww^\prime$
the extension of $w$ obtained by attaching $w^\prime$ to the right.
A tree $\mathbf{T}$ on $\mathcal{W}$ is a tree with vertices in $\mathcal{W}$ such that every child of a vertex $w \in \mathcal{W}$ is
in the set $\{wi\colon  i \in \N\}$. We denote by $\mathbb{T}$ the set of these trees. The length of a word $w \in \mathcal{W}$
will be denoted by $|w|$.
Let $\{T_w\}_{w \in \mathcal{W}}$  
be a sequence of independent random variables with law $\rho$.
For every realization of $\{T_w\}_{w \in \mathcal{W}}$ we draw an edge from a word $w$ to the 
word $wi$ (where $i \in \mathbb{N}^*$) if and only if $i \le T_w$. This is a forest;
we denote by $\tau$ 
the random GW-tree, that is the connected component of the forest containing the
root $\emptyset$. 
Consider the space $\mathbb{T}$ endowed 
with the minimal $\sigma$-algebra containing all sets $A_{\mathbf{T}}:=
\{\mathbf{T}^\prime \in \mathcal{W}\colon \mathbf{T}^\prime \supseteq \mathbf{T} \}$
where  $\mathbf{T}$ is a finite tree.
The GW-tree $\tau$ is a $\mathbb{T}$-valued random variable; we denote by  $\overline{\mu}$
its law and we call it \textit{GW-measure}.
The probability that $\tau$ is a finite tree is the smallest
fixed point $\alpha$ of $\varphi$ in $[0,1]$. Moreover, 
if $\rho(1)=1$ then $\tau^L$ is the 1-dimensional tree $\N$ and $\alpha=0$. If $\rho(1)<1$ then
$\alpha=1$ (that is, $\tau^L$ is a finite tree a.s.) if and only if $m \le 1$.

The set of vertices of a generic labelled tree is
$\mathcal{V}:=\mathcal{W} \times J$, where $J$ is the at most countable set of labels.
We consider only those trees $\Upsilon$ on $\mathcal{V}$ such that
(1) every child of $(w,j) \in \Upsilon$ is in the set $\{(wi,j^\prime) \colon i \in \N^*, j^\prime \in J\}$,
(2) if $(w,j), (w, j^\prime) \in \Upsilon$ then $j=j^\prime$,
(3) the root of $\Upsilon$ belongs
to $\{(\emptyset, j)\colon j \in J\}$.
Denote by $\mathbb{LT}$ the set of all these trees on $\mathcal{V}$.
The natural projection $\pi$ of $\mathcal{V}$ onto $\mathcal{W}$ extends to a projection
from $\mathbb{LT}$ onto $\mathbb{T}$.

Given a probability space where we have a realization of the GW-tree $\tau$ and, independently,
a family of independent $J$-valued random variables $\{N_w\}_{w \in \mathcal{W}}$, we define 
an $\mathbb{LT}$-valued random variable 
$\tau^L(\omega):=\{(w,j) \in \mathcal{V}\colon w \in \tau(\omega), j=N_w(\omega)\}$
that we call 
\textit{labelled GW-tree}. Its law,
which is uniquely determined by $\rho$ and by the laws of $\{N_w\}_{w \in \mathcal{W}}$, will be denoted by 
$\mu_{\varsigma}$ (where $\varsigma=\pr_{N_\emptyset}$ is the law of $N_\emptyset$) or simply by $\mu$.
Roughly speaking, in this special case, the random labelled GW-tree can be obtained by generating a GW-tree in the first place
and then by placing a random type (in our case, a random number of stations) on each vertex $w$ independently
with law $N_w$. The projected random variable $\pi \circ\tau^L$ is the unlabelled GW-tree $\tau$ and
$\overline{\mu}(\cdot)=\mu(\pi^{-1}(\cdot))$. Hence, given a measurable set $A \subset \mathbb{T}$, 
the measure of the set $\pi^{-1}(A) \subset \mathbb{LT}$ does not depend on  $\{N_w\}_{w \in \mathcal{W}}$.
In particular if 
$E$ is the set of finite labelled trees
then $\mu_\varsigma(E)=\alpha$ for every law $\varsigma$, since it
 is independent of $\{N_w\}_{w \in \mathcal{W}}$.

On the probability space $(\mathbb{LT}, \mu_{\varsigma})$ there is a canonical process: given $\Upsilon \in \mathbb{LT}$,
define $Z_n(\Upsilon, j)$ as the total number of vertices $(w,j) \in \Upsilon$
for all words $w \in \mathcal{W}$ of length $n \ge 0$. If $\{N_w\}_{w \in \mathcal{W}\setminus\{\emptyset\}}$
are i.i.d.~then $\{Z_n\}_{n \in \N}$ is a discrete-time MBP starting from one particle of random type
$N_\emptyset$ and $\tau^L$ is its genealogy tree.

We denote by $l(\Upsilon)$ the label of the root of $\Upsilon$;
$l$ is a random variable on $\mathbb{LT}$ with law $\varsigma$. By construction
$l$ and $\tau=\pi \circ \tau^L$ are independent. 
The measure $\mu_\varsigma$ depends on the initial distribution ${\varsigma}$;
a particular case is
${\varsigma}:=\delta_j$ where we denote the measure $\mu_{\delta_j}$ simply by $\mu_j$. 
Clearly $\mu_{\varsigma}=\sum_{j \in J}{\varsigma}(j)\mu_j$ and, for every measurable set $A \subseteq {\mathbb{LT}}$, 
$\E[A | l=j]=\mu_j(A)= \E[\tau^L \in A | Z_0=\ident_j]
$. 
The supports ${\mathbb{LT}}_j:=\{\Upsilon \in {\mathbb{LT}}\colon  l(\Upsilon)=j\}$ of the measures $\mu_j$ induces a
natural partition ${\mathbb{LT}}=\bigcup_{j \in J} {\mathbb{LT}}_j$. 

We note that given a labelled tree $\Upsilon$ and a vertex
$(w,j) \in \Upsilon$ there is a natural identification of the subtree branching from $(w,j)$ and a labelled tree in
${\mathbb{LT}}_{j}$. Analogously, every branching random subtree of a random labelled GW-tree $\tau^L$ 
can be identified with a random labelled GW-tree.
%
%

\begin{defn}\label{def:inherited}
 Given $\overline J \subseteq J$, 
a couple $(A,\widetilde A)$ of measurable sets of trees $A, \widetilde A \subseteq {\mathbb{LT}}$ is called \textit{inherited
with respect to $\overline J$} if and only if
\begin{enumerate}
 \item $\mu_j(E \setminus A)=\mu_j(A \triangle \widetilde A)=0$ for all $j \in \overline J$;
 \item if $\Upsilon \in A$ then all subtrees branching from the children of the root 
belong to $\widetilde A$.
\end{enumerate}
\end{defn}

For instance, the first condition is satisfied if $E \subseteq A = \widetilde A$; in this case,
being inherited is just a set property which does not depend on ${j \in J}$.
When $A=\widetilde A$ 
and $\overline J=J$ we simply say that $A$ is inherited
(when $J$ is a singleton this is the usual definition, see e.g.~\cite{cf:Per}).
The following lemma holds.

\begin{lem}\label{lem:BRWihenrited}
Let $\{N_w\}_{w \in \mathcal{W}}$ be an independent family of random variables such that 
$N_{w} \sim N$ for all $w \in \mathcal{W}\setminus \{\emptyset\}$ 
and define $J_N:=\{j \in J\colon  \pr(N=j)>0\}$. Suppose that $J_N \subseteq \overline J \subseteq J$.
If $(A,\widetilde A)$ is inherited w.r.~to $\overline J$
then either $\mu_j(A \triangle E)=\mu_j(\widetilde A \triangle E)=0$ for all $j \in \overline J$ or 
$\mu_j(A)=\mu_j(\widetilde A)=1$ for all $j \in J_N$.
Moreover, 
in the first case 
if $\mathrm{supp}({\varsigma}) \subseteq \overline J$, 
we have $\mu_{\varsigma}(A \triangle E)=\mu_{\varsigma}(\widetilde A \triangle E)=0$,
while in the second case if $\mathrm{supp}({\varsigma}) \subseteq J_N$, 
we have $\mu_{\varsigma}(A)=\mu_{\varsigma}(\widetilde A)=1$.
\end{lem}

Note that in general, given an inherited set $A$, it is not true that if for some $j \in J$, $\mu_j(A)>\mu_j(E)$ then
$\mu_j(A)=1$ for all $j \in J$. Indeed, suppose that $J=\{1,2\}$ and $N=2$ almost surely.
If $m>1$ then the smallest fixed point $\alpha$ of $\varphi$ 
is strictly smaller than $1$. Let $A$ be the collection of all trees which are either finite or with
root of type 2. Then $A$ is inherited and $\mu_1(A)=\alpha$, $\mu_2(A)=1$.
When $J$ is a singleton, Lemma~\ref{lem:BRWihenrited} applies to classical (unlabelled) GW-trees.

From now on, $J=\mathbb{N}$ and the \textit{environment} 
is a realization of the labelled GW-tree $\tau^L$, that is, 
a choice of the random GW-tree $\tau$ along with the number of stations at each vertex.

After the construction of the environment, we construct the probability space for our processes.  
Let $(\mathbb{LT},\mu)$ be as before and $\nu= \prod_{w \in \mathcal{W}, i \in \N^*} \pr_{R_{w,i}}$ be the product measure of the
laws of $\{R_{w,i}\}_{w \in \mathcal{W}, i \in \N^*}$ on the canonical product space $\mathcal{O}:=[0,+\infty)^{\mathcal{W} \times \N^*}$.
Henceforth we do not need the original probability space $(\Omega, \mathcal{F}, \pr)$ and we consider
$\Omega:= \mathbb{LT} \times \mathcal{O}$ endowed with the usual product $\sigma$-algebra and 
$\pr:= \mu \times \nu$. With a slight abuse of notation we denote again
by $\tau^L$, $N_w$ and $R_w$ the natural counterparts of the original random variables which are now defined on the ``new''
space $\Omega$. For every $\omega \in \Omega$ the processes evolve according to the sequence of radii 
$\{\max_{i \in \N^*} R_{w,i}(\omega) \ident_{[1,N_w(\omega)]}(i)\}_{x\in X}$
(see also Section~\ref{subsec:annealed}). 
Extinction and survival are 
measurable sets with respect to the product $\sigma$-algebra. By standard measure theory, for every event $A$ we have
\begin{equation}\label{eq:disintegration}
 \pr(A)=\int_{\mathbb{LT}} \pr(A|\tau^L=\Upsilon)\mu(\diff \Upsilon)
\end{equation}
where
$\pr(A | \tau^L=\Upsilon)=\nu( \mathbf{r} \in \mathcal{O} \colon (\Upsilon,\mathbf{r}) \in A)$
since $\tau^L(\Upsilon, \mathbf{r})\equiv\Upsilon$. 
Using equation~\eqref{eq:disintegration}, we have that
$\pr(A)=0$ (resp.~$\pr(A)=1$) if and only if $\pr(A|\tau^L=\Upsilon)=0$ (resp.~$\pr(A|\tau^L=\Upsilon)=1$) 
$\mu$-almost surely. 
It is clear that $\pr(A)>0$ if and only
if $\mu(\Upsilon\colon \pr(A|\tau^L=\Upsilon)>0)>0$. 
Quenched results, when $A$ is the event ``the process survives'', can be obtained from Lemma~\ref{lem:BRWihenrited} 
as shown by Lemma~\ref{lem:BRWihenrited2}. 

\subsection{Random environment on $\N$}
\label{subsec:N}

We discuss here the case where $\rho(1)=1$. The projection of the resulting
labelled GW-tree is 
the set of words (empty word included)
where $1$ is the only letter. We identify this projected tree with 
$\N$.
To stress the fact that we are dealing with this special case, 
we adopt a different notation. The tree $\tau^L$ can be identified with $\N^\N$-valued random variable
$\mathcal{N}$ (representing the sequence of
labels of $\tau^L$) where $\mathcal{N}(\omega):=\{N_i(\omega)\}_{i \in \N}$.
Remember that $\{N_i\}_{i \in \N}$ is a family of independent random variables (possibly with different distributions).

In this case the environment is represented by every fixed realization of $\mathcal{N}$.
Instead of $(\mathbb{LT}, \mu)$, here we use the space $(\mathbb{N}^{\mathbb{N}}, \mu)$ 
and $\mu$ is the law of $\mathcal{N}$, that is, 
the product measure $\prod_{i \in \N}\pr_{N_i}$
of the laws of $\{N_i\}_{i \in \N}$ on  $\mathbb{N}^{\mathbb{N}}$. 
The measure $\pr=\mu \times \nu$ is defined on 
$\Omega= \mathbb{N}^{\mathbb{N}}\times \mathcal{O}$.
Equation~\eqref{eq:disintegration} becomes
$
 \pr(A)=\int_{\mathbb{N}^{\mathbb{N}}} \pr(A|\mathcal{N}=\mathbf{n})\mu(\diff \mathbf{n}).
$ 
With a slight abuse of notation we denote by $\{N_x\}_{x \in X}$
the realization of the sequence $\mathcal{N}$ on $\mathbb{N}^{\mathbb{N}}\times \mathcal{O}$.


\subsection{The annealed counterpart}
\label{subsec:annealed}

%
On a deterministic graph
or on any given realization of a random graph, we associate with
our (firework or reverse firework) process
with random numbers of stations, a (firework or reverse firework)  process with one station per vertex.
Indeed, consider the process with one station on each vertex $x$
and radius $\widetilde R_x= \ident_{\{N_x \ge 1\}} \cdot \max\{R_{x,j}\colon j=1, \ldots, N_x\}$ at $x \in X$. 
We call this process, the \textit{deterministic counterpart} or \textit{annealed counterpart}
of the original process.
The annealed counterpart does not retain any information about the environment, nevertheless
the probability of annealed survival of the processes are the same.
In the following we use extensively
the cumulative distribution function of $\widetilde R_x$ which can be easily 
computed as $\ident_{[0,+\infty)}(t) G_{N_x}(\pr(R_x \le t)))$ where
$G_{N_x}(t) := \E[t^{N_x}] = \sum_{j=0}^\infty \pr(N_x=j)t^j$. As a consequence 
we have  
$\pr(\widetilde R_x<t) = 
\ident_{(0,+\infty)}(t) G_{N_x}(\pr(R_x < t))
$.
Since we assumed that $\pr(R_x<1)\in(0,1)$ then $\pr(\widetilde R_x < 1)\in(0,1)$.

\section{Firework process on $\N$}\label{sec:firework}

According to Section~\ref{subsec:N}, the environment is the random sequence $\mathcal{N}$ (where $J=\N$),
its law is $\mu$ and $\pr$ is the probability measure on $\Omega=\mathbb{N}^\mathbb{N}\times \mathcal{O}$.
We denote by $V$ the event ``the firework process survives'', that is, ``all the vertices are reached by a signal''. 
$\pr(V)$ is also the probability of survival of the annealed counterpart of the reverse firework
process; on the other hand $\pr(V|\mathcal{N}=\mathbf{n})$ is the probability of survival of the firework process
conditioned on a specific realization $\mathbf{n}$ of the sequence of  numbers of stations.
Results for the deterministic case with $k$ stations per site can be retrieved by using $G_N(t) \equiv t^k$.

\begin{teo}\label{thm:2}
 In the homogeneous case ($R_{i,j} \sim R$ and $N_i \sim N$ for all $i \in \N$, $j \in \N^*$), 
\begin{enumerate}
\item
if 
\begin{equation}\label{eq:sumprod2.5}
 \sum_{n=0}^\infty \prod_{i=0}^n G_{N}(\pr(R < i+1))=+\infty
\end{equation}
then there is extinction for $\mu$-almost all $\mathcal{N}$; 
\item if
\begin{equation}\label{eq:sumprod2}
 \sum_{n=0}^\infty \prod_{i=0}^n G_{N}(\pr(R < i+1))<+\infty
\end{equation}
then $\pr(V)>0$ and
$\mu(\mathbf{n}\colon \pr(V|\mathcal{N}=\mathbf{n})>0)=\pr(N>0)$.
\end{enumerate}
\end{teo}


%
%

Since there is no survival if there are no stations at the origin, then the probability 
that the environment $\mathcal{N}$ can sustain a surviving firework process cannot exceed $\pr(N>0)$.
The previous theorem tells us that when the above probability is not $0$ then it attains its maximum 
value.

A function $L:(0, +\infty) \to \R$ is \textit{slowly varying} if and only if for all $x>0$
($\Longleftrightarrow$ for all $x >1$) $\lim_{t \to +\infty} L(xt)/L(t)=1$. It is clear that if
a slowly varying function $L$ does not vanish on $(0, +\infty)$ then $1/L$ is a slowly varying function.
%
%
%
%
Examples of slowly varying functions are $(\ln p(\cdot))^\alpha$ where $p(\cdot)$ is a polinomial and $\alpha \in \mathbb{R}$.

The following proposition gives sufficient conditions for extinction or survival.

\begin{pro}\label{pro:homogeneous1}
 Suppose that $R_{i,j} \sim R$ and $N_i \sim N$ for all $i \in \N$, $j \in \N^*$.
 \begin{enumerate}
\item If $\liminf_{n \to \infty} n(1-G_N(\pr(R < n))) >1$ then $\pr(V)>0$ and
$\mu(\mathbf{n}\colon \pr(V|\mathcal{N}=\mathbf{n})>0)=\pr(N>0)$. 
\item If $\limsup_{n \to \infty} n(1-G_N(\pr(R < n))) <1$ then 
$\pr(V)=0$ and  $\pr(V| \mathcal{N}=\mathbf{n})=0$ for $\mu$-almost all configurations $\mathbf{n}$.
  \item If $\E[N]<+\infty$ and $\limsup_{n \to \infty} n \pr(R \ge n) < 1/\E[N]$ 
then
$\pr(V)=0$ and  $\pr(V| \mathcal{N}=\mathbf{n})=0$ for $\mu$-almost all configurations $\mathbf{n}$.
\item If $\E[N]<+\infty$ and $\E[R]<+\infty$ then $\pr(V)=0$ and  
$\pr(V| \mathcal{N}=\mathbf{n})=0$ for $\mu$-almost all configurations $\mathbf{n}$.
\item If $\liminf_{n \to \infty} n \pr(R \ge n) G^\prime_N(\pr(R < n)) >1$ then $\pr(V)>0$ and
$\mu(\mathbf{n}\colon \pr(V|\mathcal{N}=\mathbf{n})>0)=\pr(N>0)$. In particular this holds if 
$\liminf_{n \to \infty} n \pr(R \ge n) > 1/\E[N]$ (where $1/\E[N]:=0$ if $\E[N]=+\infty$).
 \end{enumerate}
\end{pro}

The following corollary gives sufficient conditions for extinction or survival where
 $N$ and $R$ play disjoint roles.

\begin{cor}\label{cor:homogeneous1}
Suppose that $R_{i,j} \sim R$ and $N_i \sim N$ for all $i \in \N$, $j \in \N^*$.

If $\pr(N>n) \sim n^{-\alpha} L(n)$ (as $n \to \infty$) for some $\alpha \in (0,1)$ and a slowly varying function $L$
then
\begin{enumerate} 
 \item $\liminf_{n \to \infty} n \pr(R \ge n)^{\alpha}L(1/\pr(R \ge n)) \Gamma(1-\alpha)  >1$ implies $\pr(V)>0$ and
$\mu(\mathbf{n}\colon \pr(V|\mathcal{N}=\mathbf{n})>0)=\pr(N>0)$; 
\item $\limsup_{n \to \infty} n \pr(R \ge n)^{\alpha}L(1/\pr(R \ge n)) \Gamma(1-\alpha)  <1$ implies 
$\pr(V)=0$ and  $\pr(V| \mathcal{N}=\mathbf{n})=0$ for $\mu$-almost all configurations $\mathbf{n}$.
\setcounter{saveenum}{\value{enumi}}
\end{enumerate}
If  $\pr(N>n) \sim c n^{-1}$  (for some $c >0$) then 
\begin{enumerate} \setcounter{enumi}{\value{saveenum}}
 \item $\liminf_{n \to \infty} n \ln(1/\pr(R \ge n)) \pr(R \ge n) >1/c$ implies $\pr(V)>0$ and
$\mu(\mathbf{n}\colon \pr(V|\mathcal{N}=\mathbf{n})>0)=\pr(N>0)$; 
\item $\limsup_{n \to \infty} n \ln(1/\pr(R \ge n)) \pr(R \ge n)   <1/c$ implies 
$\pr(V)=0$ and  $\pr(V| \mathcal{N}=\mathbf{n})=0$ for $\mu$-almost all configurations $\mathbf{n}$.
\end{enumerate}

\end{cor}

Observe that, by coupling, if $\pr(N>n) \ge n^{-\alpha} L(n)$ (resp.~$\pr(N>n) \le n^{-\alpha} L(n)$) when $\alpha \in (0,1)$
then the conclusion of Corollary~\ref{cor:homogeneous1}(1) (resp.~Corollary~\ref{cor:homogeneous1}(2))
still holds. An analogous result holds for the case $\alpha=1$.



\begin{rem}\label{rem:finiteradiussurvival}
\begin{enumerate}[(i)]
 \item 
For every fixed unbounded random variable $R$ (with finite or infinite expected value), 
there exists a random variable $N$ such that the
firework process (with $R_{i,j} \sim R$ and $N_i \sim N$ for all $i \in \N$ and $j \in \N^*$) survives. 
Let us fix $\varepsilon >0$, $\delta \in (0,1)$ and choose a variable $N$ such that
\[
 \pr \Big (N \ge \frac{\ln(1-\delta)}{\ln(\pr(R<n))} \Big ) \ge \frac{1+\varepsilon}{n\delta}
\]
for every sufficiently large $n$. Such $N$ exists since $\pr(R<n)<1$ for all $n \in \mathbb{N}$.
Indeed consider
\[
 \begin{split}
n  (1-G_N(\pr(R < n)))&= n \E[1-\pr(R<n)^N] \\
&\ge n \Big (1-\pr(R<n)^{\ln(1-\delta)/\ln(\pr(R<n))} \Big) \pr \Big (N \ge \frac{\ln(1-\delta)}{\ln(\pr(R<n))} \Big )\\
&= n\delta\, \pr \Big (N \ge \frac{\ln(1-\delta)}{\ln(\pr(R<n))} \Big ) \ge 1+\varepsilon
 \end{split}
\]
thus Proposition~\ref{pro:homogeneous1}(1) applies. 
\item
For every fixed (bounded or unbounded) $N$ such that 
$\pr(N=0)<1$, 
there exists $R$ such that 
the
firework process (with $R_{i,j} \sim R$ and $N_i \sim N$ for all $i \in \N$ and $j \in \N^*$) survives. 

Indeed, define $p_n := \inf \{t \ge 0 \colon G_N(1-t) \le 1-2/n\}$. Since $G_N(1-t)<1$ for all $t >0$ we have that
$p_n \downarrow 0$ as $n \to \infty$; moreover, by continuity, $G_N(1-p_n) \le 1-2/n$. By construction
$\liminf_{n \to \infty} n(1-G_N(1-p_n)) \ge 2$, hence if $\pr(R \ge n) =p_n$ then
Proposition~\ref{pro:homogeneous1}(1) applies.

Another proof can be derived by coupling from the following example. Let $N$ be a Bernoulli random variable with parameter
$p>0$. In this case $n(1-G_N(\pr(R<n)))=n p \pr(R \ge n)$, hence if, for instance, $\pr(R \ge n)=2 /(pn)$ then
according to Proposition~\ref{pro:homogeneous1}(1) there is survival. Since every nontrivial $N$ stochastically dominates
a Bernoulli variable with parameter $p=\pr(N>0)$, by coupling, we have survival of the homogeneous firework process
associated with $N$ and $R$ (where  $\pr(R \ge n)=2 /(pn)$).
\end{enumerate}
\end{rem}

We note that if $\pr(N=0)\in (0,1)$ in almost every realization there are connected sequences of empty vertices of
arbitrarily large length and nevertheless the process may survive  with positive probability.
This happens in particular in the Bernoulli case where we have at most one station per site.
This
proves, for instance, that the sufficient conditions of \cite[Theorem 2.2]{cf:JMZ} are not necessary.

\begin{rem}\label{rem:expectedvalues}
 Here we consider the possible behaviors of the system depending on the convergence/divergence of the expected values
$\E[N]$ and $\E[R]$.
\begin{itemize}
 \item 
If $\E[N]<+\infty$ and $\E[R]<+\infty$ then there is a.s.~extinction for almost every configuration
$\mathbf{n}$ (see Proposition~\ref{pro:homogeneous1}(4)).
\item If $\E[N]=+\infty$ and $\E[R]<+\infty$ then both survival and extinction are possible.
Indeed Remark~\ref{rem:finiteradiussurvival} proves that survival for almost every configuration $\mathbf{n}$ is possible.
On the other hand if $\pr(N \ge n) \sim n^{-\alpha}$ (for some $\alpha \in (0,1)$)
and
$\pr(R \ge n) \sim n^{-\frac{1}{\alpha}-\epsilon}$, according to Corollary~\ref{cor:homogeneous1}(2), we have a.s.~extinction
for almost every configuration $\mathbf{n}$.
\item
If $\E[N]<+\infty$ and $\E[R]=+\infty$ then both survival and extinction are possible.
Indeed fix any $N$ such that $0<\E[N]<+\infty$ and suppose that $\pr(R \ge n) \sim \alpha/n$; then, according to 
Proposition~\ref{pro:homogeneous1},
if $\alpha >1/\E[N]$ there is survival for almost every configuration $\mathbf{n}$ while
if $\alpha <1/\E[N]$ there is extinction for almost every configuration $\mathbf{n}$.
\item
If $\E[N]=+\infty$ and $\E[R]=+\infty$ then, again, both survival and extinction are possible.
Survival is easy: take $\pr(R \ge n) =p_n \vee 1/n$ (where $p_n$ is defined as in Remark~\ref{rem:finiteradiussurvival}(ii))
and apply Proposition~\ref{pro:homogeneous1}(1).

As for the extinction consider $N$ and $R$ such that $\pr(N \ge n) \sim 1/n$ as $n \to \infty$ 
and $\pr(R \ge n)=1/\big (n \ln(n) \ln(\ln(n)) \big )$ for every sufficiently large $n$. Clearly 
$\sum_{n \in \N} \pr(R \ge n )= \sum_{n \in \N} \pr(N \ge n )=+\infty$; moreover 
\[
 n \ln(\frac{1}{\pr(R \ge n)})\pr(R \ge n)=\frac{\ln(n) +\ln(\ln(n)) +\ln(\ln(\ln(n)))}{\ln(n) \cdot \ln(\ln(n))} \to 0
\]
as $n \to \infty$, hence Corollary~\ref{cor:homogeneous1}(4) applies and 
there is extinction for almost every configuration $\mathbf{n}$. 
\end{itemize}
\end{rem} 

The previous remark is summarized by the following table.
 \begin{center}
 \begin{tabular}{c|c|c}
 & $\E[N]<+\infty$ & $\E[N]=+\infty$ \\ \hline
 $\E[R]<+\infty$ & extinction & extinction/survival \\ \hline
 $\E[R]=+\infty$ & extinction/survival & extinction/survival\\
 \end{tabular}
 \end{center}

In the heterogeneous case we have a sufficient condition for survival.

\begin{teo}\label{thm:3}
 In the heterogeneous case, if
\begin{equation}\label{eq:sumprod3}
 \sum_{n=0}^\infty \prod_{i=0}^n G_{N_{i}}(\pr(R_{i} < n-i+1))<+\infty
\end{equation}
then $\pr(V)>0$. Moreover $\pr(V|\mathcal{N}=\mathbf{n})>0$ for $\mu$-almost all configurations $\mathbf{n}$.
\end{teo}

%

\section{Reverse Firework process on $\N$}\label{sec:reverse}

We use the same notation as in Section~\ref{sec:firework}.
We denote by $S$ the event ``the reverse process survives''. 
 Here we suppose that 
there is always one station at $0$. In the general case of a random number $N_0$ of station at $0$ we can condition
on the events $\{N_0=0\}$ and $\{N_0>0\}$: in the first case we have extinction, in the second one we apply the results
of this section.
As before, the results for the deterministic case with one station per site can be retrieved by using $G_N(t) \equiv t$.

\begin{teo}\label{th:reversehomogeneous}
 Let $\{R_{i,j}\}_{i,j \in \N^*}$ and $\{N_i\}_{i \in \N^*}$ be two i.i.d.~families. Define
$W:=\sum_{n \in \N} (1-G_N(\pr(R<n)))$ (or $W:=\int_0^\infty (1-G_N(\pr(R<t))) \diff t$).
\begin{enumerate}
 \item If $W=+\infty$ then $\pr(S|\mathcal{N}=\mathbf{n})=1$ for $\mu$-almost all configurations $\mathbf{n}$.
\item If $W<+\infty$ then $\pr(S|\mathcal{N}=\mathbf{n})=0$ for $\mu$-almost all configurations $\mathbf{n}$.
\end{enumerate}
\end{teo}
%

\begin{rem}\label{rem:finiteradiussurvival2}
\begin{enumerate}[(i)]
 \item If the homogeneous reverse firework process associated with $N$ and $R$ dies out almost surely, so does
the homogeneous firework process. Indeed, in this case $\sum_{n \in \N} (1-G_N(\pr(R<n)))<+\infty$, hence
$\lim_{n \to \infty} n(1-G_N(\pr(R<n)))=0$ (see Lemma~\ref{lem:lem2}) and Proposition~\ref{pro:homogeneous1}(ii)
applies. Hence if there is survival for the firework process then there is survival for the reverse firework process.
\item 
For every fixed unbounded random variable $R$ (with finite or infinite expected value), 
there exists a random variable $N$ such that the reverse
firework process (with $R_{i,j} \sim R$ and $N_i \sim N$ for all $i \in \N$ and $j \in \N^*$) survives. 
Take the example in Remark~\ref{rem:finiteradiussurvival}(i).
\item
For every fixed (bounded or unbounded) $N$ such that 
$\pr(N=0)<1$, 
there exists $R$ such that 
the reverse
firework process (with $R_{i,j} \sim R$ and $N_i \sim N$ for all $i \in \N$ and $j \in \N^*$) survives. 
Consider the example in Remark~\ref{rem:finiteradiussurvival}(ii).
\end{enumerate}
\end{rem}

Theorem~\ref{th:reversehomogeneous} gives a necessary and sufficient condition for (almost sure) survival for the 
homogeneous reverse firework process in RE,
that is, $W=+\infty$. This condition, implicitly involves both $N$ and $R$. What we want to do now,
is to find conditions involving separately $N$ and $R$.

\begin{teo}\label{th:big}
Let $\{R_{i,j}\}_{i,j \in \N^*}$ and $\{N_i\}_{i \in \N^*}$ be two i.i.d.~families ($R_{i,j} \sim R$ and $N_i \sim N$) and
$L$ a positive slowly varying function. Define $W$ as in Theorem~\ref{th:reversehomogeneous}.
 \begin{enumerate}
  \item If $\E[N]<+\infty$ then $W <+\infty$ if and only if $\E[R]<+\infty$.
\item If there exists $\varepsilon >0$, $\alpha \in (0, 1)$ 
such that $n^\alpha  \pr(N \ge n)/L(n) \ge \varepsilon$
for all $n \ge 1$ and $\int_0^{+\infty} \pr(R \ge t)^\alpha L(1/\pr(R \ge t)) \diff t =+\infty$ then $W=+\infty$.
\item If there exists $M >0$, $\alpha \in (0, 1)$ such that $n^\alpha \pr(N \ge n)/ L(n)  \le M$
for all $n \ge 1$ and $\int_0^{+\infty} \pr(R \ge t)^\alpha L(1/\pr(R \ge t)) \diff t <+\infty$ then $W<+\infty$.
\item
If there exists $M >0$ such that
$\pr(N \ge n) \le M /n$ for all $n \ge 1$
and $\int_0^{+\infty} \pr(R \ge t) \ln (1/\pr(R \ge t)) \diff t <+\infty$ then $W<+\infty$.
\item
If there exists  $\varepsilon >0$ such that
$\pr(N \ge n) \ge \varepsilon /n$ for all $n \ge 1$
and $\int_0^{+\infty} \pr(R \ge t) \ln (1/\pr(R \ge t)) \diff t =+\infty$ then $W=+\infty$.
 \end{enumerate}
\end{teo}

We note that, according to Theorem~\ref{th:big}(2--3), if there exists $\alpha \in (0, 1)$ such that 
$\pr(N \ge n) \asymp L(n)/n^\alpha$ 
then $\int_0^{+\infty} \pr(R \ge t)^\alpha L(1/\pr(R \ge t)) \diff t <+\infty$ if and only if $W<+\infty$.
Analogously, using Theorem~\ref{th:big}(4--5), if  $\pr(N \ge n) \asymp 1/n$ 
then $\int_0^{+\infty} \pr(R \ge t)  \ln (1/\pr(R \ge t)) \diff t <+\infty$ if and only if $W<+\infty$.

\begin{rem}\label{rem:expectedvalues-reverse}
 As in Remark~\ref{rem:expectedvalues}, we consider the possible behaviors of the reverse firework
process depending on the convergence/divergence of the expected values
$\E[N]$ and $\E[R]$.
\begin{itemize}
 \item 
If $\E[N]<+\infty$ and $\E[R]<+\infty$ then there is a.s.~extinction for almost every configuration
$\mathbf{n}$ (see Theorem~\ref{th:big}(1)).
\item If $\E[N]=+\infty$ and $\E[R]<+\infty$ then both survival and extinction are possible.
Indeed suppose that $\pr(N \ge n) \asymp n^{-\alpha}$ (for some $\alpha \in (0,1)$)
and
$\pr(R \ge n) \asymp n^{-\beta}$
(where by $a_n \asymp b_n$ we mean that there exist 
$m,M \in (0,+\infty)$ such that $a_n/b_n \in [n,M]$ for every sufficiently large $n$). According to the
discussion after Theorem~\ref{th:big},
if $ \beta \in (1, 1/\alpha]$  then there is survival for almost every configuration $\mathbf{n}$,
while if $\beta > 1/\alpha$ there is a.s.~extinction for almost every configuration
$\mathbf{n}$.
\item
If $\E[N]<+\infty$ and $\E[R]=+\infty$ then there is survival for almost every configuration $\mathbf{n}$ 
 (see Theorem~\ref{th:big}(1)).
\item
If $\E[N]=+\infty$ and $\E[R]=+\infty$ then, 
due to a coupling with the case $\E[N]<+\infty$ and $\E[R]=+\infty$, only survival 
is possible.
\end{itemize}
\end{rem}
 The previous remark is summarized by the following table.
 \begin{center}
 \begin{tabular}{c|c|c}
 & $\E[N]<+\infty$ & $\E[N]=+\infty$ \\ \hline
 $\E[R]<+\infty$ & extinction & extinction/survival \\ \hline
 $\E[R]=+\infty$ & survival & survival\\
 \end{tabular}
 \end{center}

In the heterogeneous case 
we have the following result.

\begin{teo}\label{th:reverseinhomogeneous}
Consider the heterogeneous reversed firework process on $\N$.
\begin{enumerate}
 \item
$\sum_{k \ge 1} (1-G_{N_{n+k}}(\pr(R_{n+k}<k)))=+\infty$ for all $n \in \N$ if and only if
$\pr(S|\mathcal{N}=\mathbf{n})=1$ for $\mu$-almost all configurations $\mathbf{n}$.
\item
If $\sum_{n \in \N} \prod_{k=1}^\infty G_{N_{n+k}}(\pr(R_{n+k}<k))<+\infty$
then $\pr(S|\mathcal{N}=\mathbf{n})>0$ for $\mu$-almost all configurations $\mathbf{n}$.
\end{enumerate}
\end{teo}

\section{Firework and Reverse Firework Processes on Galton-Watson trees}
\label{sec:GWtree}


Let us consider a GW-process with offspring distribution $\rho$.
We know that if $\rho(1)=1$ then the resulting random tree is $\N$.
This particular case has been studied in Sections~\ref{sec:firework} and \ref{sec:reverse}.
In the rest of the paper we assume that $\rho(1)<1$ and we suppose that $m:=\sum_{n \in \N} n \rho(n)>1$. The underlying random graph
will be a GW-tree generated by this process. 
Henceforth, to avoid a cumbersome notation,
we use the same symbol to denote a tree (as a graph) and its set of vertices.
Let $\varphi(z):=\sum_{n \in \N} \rho(n) z^n$ be the generating function of $\rho$ 
 and let $\alpha \in [0,1]$ be the smallest
nonnegative fixed point of $\varphi$.
When $\sum_{i=0}^k \rho(i)=1$ for some $k$ we say that the GW-tree has maximum degree $k$ or 
that it is $k$-bounded.


In this section we consider just the homogeneous case:
the set of labels $J$ is $\N$ and the random number of stations
$\{N_w\}_{w \in \mathcal{W}}$ are independent $\N$-valued random variables, 
$N_w \sim N$ for all $w \in \mathcal{W} \setminus\{\emptyset\}$.
We also assume that, in the case of the firework process, 
$N_\emptyset \sim N$, while in the case of the reverse firework process the number
of stations at the root does not matter as long as it is positive, hence we take a deterministic 
$N_\emptyset=\min\{n \in \N^*\colon \pr(N=n)>0\}$ (the minimum positive value of $N$).
In both cases the support of the law of $N_\emptyset$ is a subset of the support
of $N$. 
The environment is a random labelled GW-tree $\tau^L$ defined in Section~\ref{subsec:genericlabelledtrees} where 
the label of each vertex $w$ is the number of stations at $w$.
As in Section~\ref{subsec:genericlabelledtrees}, the law of $\tau^L$ is 
denoted by $\mu$ and $\pr$ is the probability measure on $\Omega=\mathbb{LT} \times \mathcal{O}$.
Remember that $\pr(\tau^L \textrm{ is finite})=\pr(\tau \textrm{ is finite})=\mu(E)=\mu_j(E)=\alpha<1$ for all $j \in \N$.
The radii $\{R_{w,i}\}_{w \in \mathcal{W}, i \in \N^*}$ of the stations are 
independent and identically distributed (with distribution $R$).

Two interesting particular cases are when there is one station per site ($N=1$ a.s.) and when $N$
is a Bernoulli variable. The first case can be easily retrieved from the general results by taking 
$G_N(t) \equiv t$
(see the comments along this section), while the second case is discussed in Example~\ref{exmp:bernoulli}.

The strategy is to study the annealed counterpart of the process on a GW-tree, with
one station per site and radii
$\widetilde R_w=\ident_{N_w \ge 1}\max\{R_{w,i}\colon i=1, \ldots, N_w\}$ (see Section~\ref{subsec:annealed}). We prove that under suitable conditions
the probability of survival of the annealed counterpart is $0$ (resp.~$>0$). This implies that the annealed probability of
survival of the original process is $0$ (resp.~$>0$): quenched results then follow as we explain in the following lemma
(remember the definition of $\mu$, $\overline \mu$ as the laws of $\tau^L$ and $\tau$ respectively and recall
that $l(\Upsilon)$ is the number of stations at the root of $\Upsilon \in \mathbb{LT}$).

\begin{lem}\label{lem:BRWihenrited2}
Consider a homogeneous firework process or a homogeneous reverse firework process.
 \begin{enumerate}
  \item 
If $\pr(\textrm{survival})=0$ then 
$\pr(\textrm{survival}| \tau^L=\Upsilon)=0$
for almost every $\Upsilon \in \mathbb{LT}$
(that is, for almost every tree and every sequence of stations).
 \item
If $\pr(\textrm{survival}|\tau \textrm{ is infinite})=1$ then 
$\pr(\textrm{survival}| \tau^L=\Upsilon)=1$
for almost every infinite $\Upsilon \in \mathbb{LT}$.
\item
If $\pr(\textrm{survival})>0$ and $\pr(N =0) =0$ then 
$\pr(\textrm{survival}| \tau^L=\Upsilon)>0$
for almost every infinite $\Upsilon \in \mathbb{LT}$ such that $l(\Upsilon) \ge 1$
(that is, for almost every realization of the environment such that the underlying tree is infinite and there is
at least one station at the root). 
\item
If $\pr(\textrm{survival})>0$ and $\pr(N =0) >0$ then 
$\pr(\textrm{survival}| \tau=\mathbf{T}, N_\emptyset=n)>0$
for almost every  $(\mathbf{T}, n) \in \mathbb{T} \times \N$
such that the tree $\mathbb{T}$ is infinite and $n \ge 1$.
 \end{enumerate}
Moreover,
$\mu(\Upsilon \colon \pr(\textrm{survival}| \tau^L=\Upsilon)>0)$ 
and $\overline\mu\times\pr_{N_\emptyset}((\mathbf T,n)\colon \pr(\textrm{survival}| \tau=\mathbf T,N_\emptyset=n)>0)$
are both either $0$ or $(1-\alpha)\pr({N_\emptyset}\ge 1)$.
\end{lem}

%
%
%
The third case of the previous lemma applies, for instance, to any annealed counterpart
of a process (since in the annealed counterpart there is one station per vertex). 

Note the difference between ``$\pr(\mathrm{survival}|\tau^L=\Upsilon)>0$ 
for almost every $\Upsilon \in \mathbb{LT}$ such that
the underlying tree is infinite and $l(\Upsilon) \ge 1$'' and ``$\pr(\mathrm{survival}|\tau=\mathbf{T}, N_\emptyset =n)>0$ for almost every infinite tree $\mathbf{T}\in \mathbb{T}$
and for $\pr_{N_\emptyset}$-almost all $n \ge 1$''. 
In both cases there is at least one station at the root, but
the second assertion is
weaker than the first one
since in the former the tree and the number of stations at each vertex are fixed while in
the latter just the tree and the number of stations at the root are fixed. 
Indeed we show that, when $\pr(N=0)>0$, the second assertion may hold while the first  does not 
 (see Example~\ref{exmp:counterexample}).


Finally, since there is no survival if there are no stations at the root or the tree is finite, then
the probability that the environment can sustain a surviving process cannot exceed the probability of
the event ``the tree is infinite and there are stations at the root'', that is,
$(1-\alpha)\pr({N_\emptyset}\ge 1)$. Lemma~\ref{lem:BRWihenrited2} tells us that
when the probability that the environment can sustain a surviving process is not $0$ then it
attains the maximum admissible value.

\begin{teo}\label{th:GW-Homogeneous-Firework2}
Consider a homogeneous firework process. Define 
$\Phi(t):=G_N(\pr(R<1))+\sum_{n=1}^\infty (G_N(\pr(R<n+1))-G_N(\pr(R<n))) t^n \in [0, +\infty]$ ($t \in [0,+\infty)$).
\begin{enumerate}
 \item If 
$\Phi(m)-1>\Phi(0)=G_N(\pr(R<1))$ and $\pr(N=0)=0$ then for the firework process there is survival with positive probability
 for almost every realization of the environment such that the underlying tree is infinite
 and there is at least one station at the root.
 \item If 
$\Phi(m)-1>\Phi(0)=G_N(\pr(R<1))$ and $\pr(N=0)>0$ then for the firework process 
$\pr(\mathrm{survival}|\tau=\mathbf{T}, N_\emptyset =n)>0$ for almost every $(\mathbf{T},n) \in \mathbb{T} \times \N$ such that
$\mathbf{T}$ is an infinite (unlabelled) tree and $n \ge 1$.
\item
 If the GW-tree is $k$-bounded and 
$\Phi(k)-1\le 1-1/k$
then the firework process becomes
extinct a.s.~ for almost every 
realization of the environment.
\end{enumerate}
\end{teo}

It is clear, from the previous theorem, that the radius of convergence of $\Phi$ will play an important role.
Elementary computations show that $\limsup_{n \to \infty} \sqrt[n]{G_N(\pr(R < n+1))-G_N(\pr(R < n))}= \\
\limsup_{n \to \infty} \sqrt[n]{1-G_N(\pr(R < n))}$.
%
%
%
Moreover, in the case $N=1$ a.s., 
$\Phi(t):=\E[t^{\lfloor R\rfloor}]=
\sum_{n=0}^\infty \pr(n \le R < n+1) t^n$.

\begin{cor}\label{cor:GW-Homogeneous-Firework}
Define 
$\overline m_c:=\sup\{m >0 \colon \textrm{ the firework process dies out a.s.}\}$. 
We have that \break
$\limsup_{n \to \infty} \sqrt[n]{1-G_N(\pr(R < n))}=1$ implies $
\overline m_c= 1$.
\end{cor}
Consider the case $N=1$ almost surely. In this case, 
if $\E[R^a]=+\infty$ for some $a >0$ we have $m_c=1$. An example where $\E[R]<+\infty$ and
$\overline m_c=1$ is given by any law $R$ such that
$\pr(n \le R < n+1) \asymp n^{-\alpha}$ ($\alpha >2$).

The next result deals with the behavior of the reverse homogeneous firework process.
By definition, $\prod_{j=1}^0 \alpha_n:=1$ for every sequence $\{\alpha_n\}_{n \in \N}$.

\begin{teo}\label{th:GW-Homogeneous-reverse-Firework2}
Consider a homogeneous reverse firework process. 
Let $\phi_1(m):=\sum_{n=1}^\infty (1-G_N(\pr(R < n))){m^n}$ and 
$\phi_2(m):=\sum_{i=1}^\infty (1-G_N(\pr(R < i))) m^i \prod_{j=1}^{i-1} G_N(\pr(R<j))$.
The following hold
\begin{enumerate}
\item if 
$\phi_1(m)=+\infty$ then there is survival with probability $1$ for the reverse firework process
for almost all realizations of the environment such that the underlying tree is infinite;
\item if $\pr(N =0)=0$,
$\phi_1(m)<+\infty$ and 
$\phi_2(m) >1$  
then there is survival with positive probability (strictly smaller than $1$) for the reverse firework process
for almost all realizations of the environment such that the underlying tree is infinite;
\item if $\pr(N =0)>0$,
$\phi_1(m)<+\infty$ and 
$\phi_2(m) >1$  
then 
$\pr(\mathrm{survival}|\tau=\mathbf{T} 
)\in(0,1)$ for almost every infinite (unlabelled) tree 
$\mathbf{T}\in \mathbb{T}$;
 \item if $\phi_1(m)<+\infty$ and 
$\phi_2(m) \le 1$
then there is a.s.~extinction for the reverse firework process
for almost all 
realizations of the environment;
\end{enumerate}

\end{teo}

We note that
when $N=1$ a.s.~then  $\phi_1(m)$ and $\phi_2(m)$ become 
$\sum_{n=1}^\infty {m^n}\pr(R \ge n)$ and
$\sum_{n=1}^\infty m^n \pr(R \ge n) \prod_{j=1}^{n-1} \pr(R<j)$ respectively.

\begin{cor}\label{cor:GW-Homogeneous-reverse-Firework2}
Define $M_c:=1/\limsup_{n \to \infty} \sqrt[n]{1-G_N(\pr(R < n))}$.
There exists a critical value $m_c \in [1, +\infty)$, 
$m_c \le M_c$
such that 
\begin{enumerate}
 \item 
$m <m_c$
implies a.s.~extinction 
for almost all 
realizations of the environment;
\item
$m_c<m<M_c$ and $\pr(N=0) =0$ implies survival with positive probability 
for almost all realizations of the environment such that the underlying tree is infinite;
\item
$m_c<m<M_c$ and $\pr(N=0) >0$  implies survival with positive probability 
for almost every infinite (unlabelled) tree;
\item $M_c<m$ implies survival with probability 1 
for almost all realizations of the environment such that the underlying tree is infinite.
\end{enumerate}
Moreover, if $m=m_c<M_c$ then there is a.s.~extinction for almost all 
realizations of the environment.
\end{cor}

Again the previous corollary applies, in particular, in the case $N=1$ a.s.~by using
$\pr(R \ge n)$ instead of $1-G_N(\pr(R < n)$.

\begin{rem}\label{rem:phasetransition2}
If $\phi_1(1)=+\infty$, 
(that is, $\E[R]=+\infty$ if $N=1$ a.s.)
then $M_c= m_c=1$ and there is survival with probability 1 for
every realization of the environment such that the underlying tree is infinite (when $m>1$). 

If $M_c=+\infty$ then the probability of survival is smaller than $1$ for almost every infinite (unlabelled) tree 
(it might be $0$ of course).

If $M_c \in(1, +\infty]$ and 
$
\phi_1(M_c)= +\infty$  
then $m_c<M_c$, thus there is a.s.~extinction for almost all 
realizations of the environment when $m=m_c$ (see the details in Section~\ref{sec:proofs}). 
\end{rem}

Note that according to Theorems~\ref{th:GW-Homogeneous-Firework2} and \ref{th:GW-Homogeneous-reverse-Firework2}
the firework and reverse firework processes on a non-trivial GW-tree can survive even if the radius $R$ is bounded. 
Indeed, if $\pr(N=0)<1-1/m$ and $\pr(R <1)$ is sufficiently small then we have $1-G_N(\pr(R<1)) > 1/m$ and the
conditions for survival given in Theorems~\ref{th:GW-Homogeneous-Firework2} and \ref{th:GW-Homogeneous-reverse-Firework2} hold.
According to Theorems~\ref{thm:2}(1) and \ref{th:reversehomogeneous}(2), this is not possible on $\N$.
As an example we discuss the Bernoulli case.

\begin{exmp}\label{exmp:bernoulli}
 We consider $N \sim B(p)$, a Bernoulli variable with parameter $p\in(0,1)$. In this framework, each vertex
is occupied by one station independently with probability $p$ and unoccupied with probability $1-p$.
Let $R_w \sim R$ for all vertices $w \in \mathcal{W}$.
Clearly $G_N(t)=pt +1-p$. 

Consider the firework process. Applying Theorem~\ref{th:GW-Homogeneous-Firework2}, we have that
\begin{enumerate}
 \item $\sum_{n=1}^\infty \pr(n \le R < n+1) m^n>1/p$ implies survival with positive probability 
 for almost every infinite (unlabelled) tree if there is at least one station at the root;
\item GW-tree $k$-bounded and
$\sum_{n=1}^\infty \pr(n \le R < n+1) k^n - \pr(R \ge 1) \le (k-1)/(kp)$ implies
a.s.~extinction for almost every 
realization of the environment.
\end{enumerate}

Consider the reverse firework process. According to Theorem~\ref{th:GW-Homogeneous-reverse-Firework2}, there is 
survival with probability $1$ 
(for almost all realizations of the environment such that the underlying tree is infinite)
if and only if $\sum_{n=1}^\infty {m^n}\pr(R \ge n)=+\infty$ (note that this condition does not depend
on $p>0$).
If $\sum_{n=1}^\infty {m^n}\pr(R \ge n)<+\infty$ then the probability of survival is strictly less than 1
for every (unlabelled) tree. In particular 
there is a.s.~extinction 
for almost all 
realizations of the environment if and only if $\sum_{n=1}^\infty m^n \pr(R \ge n) \prod_{j=1}^{n-1} (1-p\pr(R \ge j)) \le 1/p$. 
We observe that $M_c$ does not depend on $p$.
As for the critical value $m_c= m_c(p)$ it is not difficult to show
that
$p \mapsto m_c(p)$ is
nonincreasing 
and continuous. 
In particular $\lim_{p \to 0} m_c(p)=M_c$. 
Indeed, if $M_c=1$ there is nothing to prove. Otherwise, suppose that $1<r<M_c$; since
$\lim_{p \to 0} p \sum_{n=1}^\infty m^n \pr(R \ge n) \prod_{j=1}^{n-1} (1-p\pr(R \ge j))=0$ according to 
Theorem~\ref{th:GW-Homogeneous-reverse-Firework2} and
Corollary~\ref{cor:GW-Homogeneous-reverse-Firework2}
we have that
$\liminf_{p \to 0} m_c(p) \ge r$.
\end{exmp}

%
%
%
When $\pr(N=0)>0$ and $\pr(\mathrm{survival})>0$ it could happen that 
the firework and reverse firework processes die out almost surely on a $\mu$-positive set of infinite labelled trees with at least
one station at the root. This implies that, even when the probability of survival is positive, it is not possible, in general, to
have positive probability of survival for almost every realization of the environment (but only for almost every realization of
the unlabelled GW-tree) as the following example shows.

\begin{exmp}\label{exmp:counterexample}
Suppose that $\pr(N=0) \in (0,1)$, $m >1$ and $\pr(R \le M)=1$ (for some $M \in \N^*$). Under these conditions, it is an easy
exercise to prove that the set of infinite labelled trees $\Upsilon$ with $l(\Upsilon) \ge 1$ such that every vertex 
between distance $1$ and $M$ from the root has label 0, has always $\mu$-positive measure. 
On this set the firework and reverse firework processes die out
almost surely. Nevertheless, according to Theorem~\ref{th:GW-Homogeneous-Firework2} for every sufficiently large $m$
the firework process survives with (strictly) positive probability, Similarly, according to Theorem~\ref{th:GW-Homogeneous-reverse-Firework2}
there is a positive (strictly smaller than 1) probability of survival for the reverse firework process  
(note that, if $\pr(R \le M)=1$, then $\phi_1(m) <+\infty$ for all $m$).
This happens, for instance, in the Bernoulli model described in Example~\ref{exmp:bernoulli}.
\end{exmp}

\section{Proofs}\label{sec:proofs}


\begin{proof}[Proof of Lemma~\ref{lem:BRWihenrited}]
In the following we consider the canonical MBP $\{Z_n\}_{n \in \N}$ defined on the probability space $\mathbb{LT}$
endowed with the probability measure $\mu_{\varsigma}$. Hence $\tau^L(\Upsilon)=\Upsilon$ for all $\Upsilon \in \mathbb{LT}$.
Remember the definition of the generating function $\varphi(z)=\sum_{n \in \N} \rho(n) z^n$ and its smallest fixed point
$\alpha \in [0,1]$. Define $S_J:=\{g \in \N^J\colon |g|:=\sum_{j \in J} g(j)<+\infty\}$ and denote by $\pr_N$ the common law
of $\{N_w\}_{w \in \mathcal{W} \setminus\{\emptyset\}}$.
Given an inherited couple $(A,\widetilde A)$ where $A, \widetilde A \subset \mathbb{LT}$, 
if the expected value $m$ of the law $\rho$ satisfies $m \le 1$ then $\mu_j(A) \ge \mu_j(E)= \alpha=1$ for all $j \in \overline J$,
and the conclusion follows.

Suppose that $m>1$. 
By conditioning on $Z_1$, using the conditional independence and the Markov property for labelled GW-trees we have, 
for all $j \in  \overline J$,
\[
\begin{split}
 \mu_j(\widetilde A)&=\mu_j(A)=\E^{\mu_j}[\ident_{A }]=\E^{\mu_j}[\mu_j(\Upsilon \in A|Z_1)] 
 \le \E^{\mu_j}[\mu_j(\Upsilon\colon \Upsilon^{(1)} \in \widetilde A, \ldots ,\Upsilon^{(|Z_1|)} \in \widetilde A|Z_1)] \\
&=\sum_{n \in \N}\,\,  \sum_{g \in S_J \colon |g|=n} \,
\mu_j(\Upsilon\colon \Upsilon^{(1)} \in \widetilde A, \ldots ,\Upsilon^{(|g|)} \in \widetilde A|Z_1=g) \mu_j(Z_1=g) \\
&= \sum_{n \in \N}\,\,  \sum_{g \in S_J \colon |g|=n} \,\mu_j(Z_1=g) 
\prod_{i=1}^{|g|} \mu_j(\Upsilon\colon \Upsilon^{(i)} \in \widetilde A|Z_1=g) \\
&= \sum_{n \in \N}\,\,  \sum_{g \in S_J \colon |g|=n} \,\, \mu_j(Z_1=g)
\prod_{{j^\prime} \in J} \mu_{{j^\prime}}(\Upsilon \in \widetilde A)^{g({j^\prime})} \\
&=
\sum_{n \in \N}\,\, \sum_{g \in S_J \colon |g|=n} \,\, \rho(|g|)\frac{|g| !}{\prod_{j^\prime \in J} g(j^\prime)!}
\prod_{j^\prime \in J} \pr_N(j^\prime)^{g({j^\prime})}\prod_{{j^\prime} \in J} \mu_{{j^\prime}}(\widetilde A)^{g({j^\prime})} 
 \\
&= \sum_{n \in \N} \rho(n) \Big ( \sum_{j^\prime \in J} \mu_{{j^\prime}}(\widetilde A) 
\pr_N(j^\prime)\Big )^n
= \varphi \Big ( \sum_{j^\prime \in J_N} \mu_{{j^\prime}}(\widetilde A) \pr_N(j^\prime)\Big ),
\\\end{split}
\]
where $|Z_1|:=\sum_{j \in J}Z_1(j)$ is the total number of descendants of the root.

 Take a generic $h \in [0,1]^J$ and suppose that $\alpha \le h(j) \le \varphi \big (\sum_{j^\prime \in J_N} h(j^\prime) \pr_N(j^\prime) \big )$
for all $j \in \overline J$. Define $\overline h:=\sum_{j^\prime \in J_N} h(j^\prime) \pr_N(j^\prime)$.
Hence
\[
 \bar h \le \sup_{j \in \overline J} h(j) \le 
\varphi(\bar h). 
\]
Thus, either  $\bar h=\alpha$ or $\bar h=1$. This implies 
$\bar h \le \sup_{j \in \overline J} h(j) \le \varphi(\bar h)=\bar h$.
Since $h(j) \in [\alpha,1]$ for all $j \in \overline J$, then $\bar h=\alpha$
implies $h(j)=\alpha\equiv \mu_j(E)$ for all $j \in \overline J$, while $\bar h=1$ implies $h(j)=1$ for all $j \in J_N$.

By Definition~\ref{def:inherited}, $\mu_j(A)=\mu_j(\widetilde A)$, $\mu_j(E \setminus A)=\mu_j(E \setminus \widetilde A)=0$ 
and $\mu_j(A \setminus E)=\mu_j(\widetilde A \setminus E)$  for all $j \in \overline J$.
If $h(j):=\mu_j(\widetilde A)$ for all $j \in \overline J$ then we have that either $\mu_j(\widetilde A)=\mu_j(A)=1$ for all $j \in J_N$
or $\mu_j(\widetilde A)=\mu_j(A)=\mu_j(E)$ for all $j \in \overline J$ and the conclusion follows.

If ${\varsigma}$ (the law of $N_\emptyset$)  
satisfies $\mathrm{supp}({\varsigma}) \subseteq J^\prime$ then $\mu_{\varsigma}=\sum_{j \in J^\prime} {\varsigma}(j) \mu_j$, whence,
for every measurable $B \subseteq \mathbb{LT}$,
$
\mu_{\varsigma}(B)=\sum_{j \in J^\prime} {\varsigma}(j) \mu_j(B) 
$
which yields the conclusion (by taking $J^\prime=\overline J$, $B:=A \triangle E$ and $J^\prime=J_N$, $B:=A$).
\end{proof}

\begin{lem}\label{lem:1}
 If $\{t_{i,n}\}_{i,n \in \N, i \le n}$ is an arbitrary nonnegative sequence  and
\begin{equation}\label{eq:sumprod}
 \sum_{n=0}^\infty \prod_{i=0}^n G_{N_i}(t_{i,n})<+\infty
\end{equation}
then
\[
 \pr \left ( \sum_{n=0}^\infty \prod_{i=0}^n t_{i,n}^{N_i} <+\infty
\right )=1.
\]
In particular if equation \eqref{eq:sumprod} holds
\[
 \pr \left ( \sum_{n=0}^\infty \prod_{i=0}^n t_{i,n}^{N_i} <+\infty \Big | \mathcal{N}=\mathbf{n}
\right )=1
\]
for $\mu$-almost all $\mathbf{n}=\{n_i\}_{i \in \N} \in {\mathbb{N}^{\mathbb{N}}}$.
\end{lem}

\begin{rem}\label{rem:tail}
Suppose that 
$t_{i,n} \to 1$ as $n \to \infty$ for all $i \in \N$.
 We observe that the event $\{\sum_{n=0}^\infty \prod_{i=0}^n t_{i,n}^{N_i}< +\infty\}$ is a tail event with respect to
$\{\sigma(N_i\colon i \le n)\}_n$. Hence its probability is either 0 or 1. Indeed, for all $i_0 \in \N$,
$\prod_{i=0}^n t_{i,n}^{N_i} \sim \prod_{i=i_0}^n t_{i,n}^{N_i}$ as $n \to \infty$ hence
$\{\sum_{n=0}^\infty \prod_{i=0}^n t_{i,n}^{N_i}< +\infty\}=\{\sum_{n=i_0}^\infty \prod_{i=i_0}^n t_{i,n}^{N_i}< +\infty\}$.
\end{rem}

\begin{proof}[Proof of Lemma~\ref{lem:1}]
 Let $\xi:= \sum_{n=0}^\infty \prod_{i=0}^n t_{i,n}^{N_i}$. Note that, by the Monotone Convergence Theorem and
using the independence of $\{N_i\}_{i \in \N}$,
\[
\begin{split}
 \E \left [ \xi \right ] &=  \E_\mu \left [ \xi \right ] = \int \sum_{n=0}^\infty \prod_{i=0}^n t_{i,n}^{N_i} \diff \mu
= \sum_{n=0}^\infty \E_\mu \left [ \prod_{i=0}^n t_{i,n}^{N_i} \right ] \\
&= \sum_{n=0}^\infty \prod_{i=0}^n G_{N_i}(t_{i,n}) < +\infty
\end{split}
\]
thus $\pr(\xi<+\infty)\equiv\mu(\xi <+\infty)=1$. The last equality is an easy consequence of equation~\eqref{eq:disintegration}.
\end{proof}

\begin{proof}[Proof of Theorem~\ref{thm:2}]
We investigate the behavior of the deterministic counterpart of our process.
Since $R_{i,j} \sim R$ and $N_i \sim N$ for all $i,j$ then $\widetilde R_{i} \sim \widetilde R$
where 
\[
 \pr(\widetilde R < i)= \sum_{j \in \N} \pr(R <i)^j \pr(N=j)=G_N(\pr(R <i)).
\]
By  \cite[Theorem 2.1]{cf:JMZ} 
the a.s.~extinction of the deterministic counterpart is equivalent to 
\[
 \sum_{n=1}^\infty \prod_{i=0}^n \pr(\widetilde R < i+1)=+\infty.
\]
Hence equation~\eqref{eq:sumprod2} is equivalent to $\pr(V)=0$ which, in turn, 
is equivalent to $\pr(V| \mathcal{N}=\mathbf{n})=0$ for $\mu$-almost all configurations $\mathbf{n}$.

We are left to prove that $\pr(V)>0$ implies $\mu(\mathbf{n}  \in {\mathbb{N}^{\mathbb{N}}} \colon \pr(V|\mathcal{N}=\mathbf{n})>0)=\pr(N>0)$.
Note that $\pr(V)=\pr(V| N_0>0) \pr(N_0>0)$
since $\pr(V| N_0=0)=0$; moreover $\pr(V| N_0>0)>0$ if and only if 
$\pr(V)>0$. We condition now on the event $\{N_0>0\}$.
We denote by $\overline{{\mathbb{N}^{\mathbb{N}}}}$ the space $\{\mathbf{n}\colon n_0>0\}$ and by $\overline \mu$ the measure $\mu$
conditioned on $\overline{{\mathbb{N}^{\mathbb{N}}}}$ (clearly $\mu(\overline{{\mathbb{N}^{\mathbb{N}}}})=\pr(N_0>0)$). Observe that $\mu(\mathbf{n}\colon  
\limsup_i n_i >0)=1$,
$\mu(\mathbf{n}\colon n_0>0)=\pr(N_0>0)$ and that
the variables $\{N_i
\}_{i \ge 1}$ are independent 
with respect to the conditioned
probability $\pr(\cdot| N_0>0)$.
Denote by 
\[
 W_{k}:=\{\mathbf{n} \in \overline{{\mathbb{N}^{\mathbb{N}}}}\colon  \pr(\textrm{``}V \textrm{ starting from }\mathbf{n}{(k)} \textrm{ stations at }
k\textrm{''}| \mathcal{N}=\mathbf{n})>0 \}
\]
the set of realizations of the environment such that the firework process starting from 
vertex $k$ (and moving rightwards) survives with  positive probability. 
Note that $\mu(W_0)=\mu(\mathbf{n} \in {\mathbb{N}^{\mathbb{N}}} \colon  \pr(V| \mathcal{N}=\mathbf{n})>0 )$.
Clearly $W_k \in \sigma (N_i\colon i \ge k)$
where, in this case, $\{N_i\}_{i \in \N}$ is the canonical realization of $\mathcal{N}$ on ${\mathbb{N}^{\mathbb{N}}}$.

If for a fixed sequence in $\overline{{\mathbb{N}^{\mathbb{N}}}}$  there is survival for the firework process 
starting from $n_{i_0}$ stations at $i_0$ 
then,
 by the FKG inequality,  there is survival starting from every $i \le i_0$ 
(note that 
if equation~\eqref{eq:sumprod2.5} holds then the radii are unbounded variables, hence
the event $\{R_0>i_0\}$ has a positive 
  probability). Whence $W_k \supseteq W_{k+1}$ for all $k \in \N$.

Moreover, if there is survival for a sequence $\mathbf{n}$ then for all $j \in \N$ 
  there exists $i_0 \ge j$ such that there is survival starting from
$\mathbf{n}{(i_0)}$ particles at $i_0$. Hence 
$W_0 \subseteq \limsup_k W_k$ which implies that $W_0=\limsup_k W_k=W_i$ for all $i \in \N$.
Hence $W_0$ is a tail event, namely it belongs to $\bigcap_{k \in \N} \sigma (N_i\colon i \ge k)$.
Thus
$\overline \mu(W_0)$ is either $0$ or $1$. This implies that 
$\mu(W_0)$ is either $0$ or $\pr(N_0>0)$.

Equation~\eqref{eq:sumprod2.5} implies $\pr(V)>0$ or, equivalently, 
$\mu(\mathbf{n} \in {\mathbb{N}^{\mathbb{N}}} \colon \pr(V| \mathcal{N}=\mathbf{n})>0 )>0$.
Since $\mu(W_0)=\mu(\mathbf{n} \in {\mathbb{N}^{\mathbb{N}}} \colon \pr(V| \mathcal{N}=\mathbf{n})>0 )$
we have that $\mu(\mathbf{n} \in {\mathbb{N}^{\mathbb{N}}} \colon \pr(V| \mathcal{N}=\mathbf{n})>0 )=\pr(N_0>0)$.
\end{proof}

Before proving Proposition~\ref{pro:homogeneous1} we need another lemma.

\begin{lem}\label{lem:lem2}

\begin{enumerate}
 \item Consider a sequence  $\{x_i\}_{i\in \mathbb{N}}$ of nonnegative real numbers.
If $\{y_i\}_{i \in \mathbb{N}}$ is a nondecreasing
sequence such that $y_0>0$ 
and $\limsup_n \sum_{i=0}^n x_i / y_n>0$
then
\[
 \sum_{i=1}^\infty x_i =+\infty \Longleftrightarrow \sum_{i=0}^\infty \frac{x_i}{y_i} =+\infty.
\]
In particular, if $x_0>0$, 
\[
 \sum_{i=1}^\infty x_i =+\infty \Longleftrightarrow \sum_{i=0}^\infty \frac{x_i}{\sum_{j=0}^i x_j} =+\infty.
\]
\item
If $\{\alpha_i\}_{i \in \mathbb{N}}$ is a 
nondecreasing sequence of strictly positive real numbers such that
$\liminf_{n} n/\alpha_n>0$
then for every nonincreasing, nonnegative sequence $\{z_i\}_{i \in \mathbb{N}}$
such that $\sum_{n=0}^\infty z_n <+\infty$ we have $\lim_{n \to \infty} z_n \alpha_n =0$.
\end{enumerate}
\end{lem}

\begin{proof}
\begin{enumerate}
 \item 
Observe that $x_i/y_i \le x_i/y_0$, 
this implies $\sum_{i=0}^\infty {x_i}/{y_i} \le \sum_{i=1}^\infty x_i/y_0$. Whence
$\sum_{i=0}^\infty {x_i}/{y_i} =+\infty$ implies $\sum_{i=1}^\infty x_i =+\infty$.

Conversely, suppose that $\sum_{i=1}^\infty x_i =+\infty$. Fix
$\delta \in(0,\limsup_n \sum_{i=0}^n x_i / y_n)$. For all $m \in \mathbb{N}$ there exists $n >m$ such that
$\sum_{i=m}^n x_i/y_n > \delta$. By induction we can find a strictly increasing sequence
$\{n_j\}_{j \in \mathbb{N}}$ of natural numbers such that $\sum_{i=n_j+1}^{n_{j+1}} x_i/y_{n_{j+1}}>\delta$.
Clearly
\[
 \sum_{i=0}^\infty \frac{x_i}{y_i}=\sum_{j=0}^\infty \sum_{i=n_j+1}^{n_{j+1}} \frac{x_i}{y_i} \ge
\sum_{j=0}^\infty \sum_{i=n_j+1}^{n_{j+1}} \frac{x_i}{y_{n_{j+1}}}=+\infty.
\]
\item
By contradiction, suppose that, for some increasing sequence $\{n_j\}_{j \in \mathbb{N}}$ and  $\delta>0$, we have 
$z_{n_j} \alpha_{n_j} \ge \delta$ for all
$j \in \mathbb{N}$. 
Then for all $n \in (n_j,n_{j+1}]$ we have $z_n \ge z_{n_{j+1}} \ge \delta/\alpha_{n_{j+1}}$.
Thus,  $\sum_{n=0}^\infty z_n \ge \sum_{j=0}^\infty (n_{j+1}-n_{j})\delta/\alpha_{n_{j+1}}$.
If we define $x_j:=n_{j+1}-n_j$ and $y_j:=\alpha_{n_{j+1}}$ then $\sum_{j \in \N} x_j=+\infty$ and
$\limsup_{j \to \infty} \sum_{i=0}^j x_i /y_j \ge \liminf_{n \to \infty} n/\alpha_n >0$. 
Thus, according to (1), 
we have
\[
 \sum_{n=0}^\infty z_n \ge \sum_{j=0}^\infty (n_{j+1}-n_{j})\delta/\alpha_{n_{j+1}}=\delta \sum_{j=0}^\infty \frac{x_{j}}{y_j}=
+\infty
\]
which contradicts
$\sum_{n=0}^\infty z_n <+\infty$.
\end{enumerate}

\end{proof}

A particular case in Lemma~\ref{lem:lem2} is
 $\alpha_n=n$.

\begin{proof}[Proof of Proposition~\ref{pro:homogeneous1}]
\begin{enumerate}
\item 
In order to prove that equation~\eqref{eq:sumprod2} holds we use the Kummer's test.
According to Kummer's result, if $\{a_n\}_{n \in \N}$ and  $\{p_n\}_{n \in \N}$ are two positive sequences,
$\alpha:=\liminf_n (p_n a_n/a_{n+1}-p_{n+1})$ and $\beta:=\limsup_n (p_n a_n/a_{n+1}-p_{n+1})$ then
$\alpha >0$ implies $\sum_n a_n <+\infty$, while $\sum_n 1/p_n=+\infty$ and $\beta <0$ implies  $\sum_n a_n =+\infty$.

If we take $p_n=n+2$ and $a_n =\prod_{i=0}^n G_N(\pr(R < i+1))$ then $\alpha >0$ hence 
$\pr(V)>0$ (and  $\mu(\mathbf{n}\colon \pr(V|\mathcal{N}=\mathbf{n})>0)>0$).

 \item As in the previous case, it follows immediately from Theorem~\ref{thm:2} and Kummer's Test
with $p_n=n+2$.
 
 \item If $\E[N]<+\infty$ then $1-G_N(\pr(R < n))\sim \E[N] \pr(R \ge n)$ as $n \to \infty$.
The result follows from $(2)$.

\item 
Trivially, $\E[R]<+\infty$ if and only if $\sum_{n \in \mathbb{N}} \pr(R \ge n) <+\infty$.
From Lemma~\ref{lem:lem2}(2) with $\alpha_n=n$ and $z_n= \pr(R \ge n)$ we have that 
$\E[R]<+\infty$ implies $\lim_{n \to \infty} n \pr(R \ge n) =0 < 1/\E[N]$ and the previous part of the theorem
applies.

\item In this case, since $G^\prime$ is a power series with nonnegative coefficients,  
$1-G_N(\pr(R < n)) \ge \pr(R \ge n) G^\prime_N(\pr(R < n))$. Thus, the result follows from $(1)$.
\end{enumerate}
\end{proof}

\begin{proof}[Proof of Corollary~\ref{cor:homogeneous1}]
The main idea of the proof is to compute a suitable asymptotic estimate $1-G_N(x)\sim (1-x)f(x)$ as $x \to 1^-$.
\begin{enumerate}
\item
Note that, using the Cauchy product of power series,  $(1-G_N(x))/(1-x)= \sum_{n=0}^\infty \pr(N >n) x^n$ for all 
$|x| <1$.
From a well-known Tauberian theorem (see e.g.~\cite[Theorem 9]{cf:B74} or \cite[Sec.~XIII.5, Theorem 5]{cf:FII}),
we have that $\pr(N >n) \sim n^{-\alpha} L(n)$ if and only if $\sum_{n=0}^\infty \pr(N >n) x^n \sim \Gamma(1-\alpha) 
L(1/(1-x))/(1-x)^{1-\alpha}$
as $x \to 1^-$. Hence $n(1-G_N(\pr(R<n)))\sim n \pr(R \ge n)^{\alpha}L(1/\pr(R \ge n)) \Gamma(1-\alpha)$ and 
Proposition~\ref{pro:homogeneous1}(1) yields the
result.
\item It follows analogously from Proposition~\ref{pro:homogeneous1}(2).
\item
If $\pr(N \ge n) \sim c n^{-1}$, since $\sum_{n=0}^\infty \pr(N \ge n)$ is divergent, then $\sum_{n=0}^\infty \pr(N \ge n) x^n \sim 
c \ln (1/(1-x))$.
Indeed, for any $\varepsilon >0$, there exists $\overline n$ such that for all $n\ge \overline n$, we have  $\pr(N\ge n) \in (c(1-\varepsilon)/n, 
c(1+\varepsilon)/n)$. Hence, eventually  as $x \to 1^-$,
\[ 
\begin{split}
\frac{\sum_{n=0}^\infty \pr(N \ge n) x^n}{c \ln (1/(1-x))} &= \frac{\sum_{n=0}^\infty \pr(N \ge n) x^n}{c \sum_{n=1}^{\infty} x^n/n} = 
\frac{\sum_{n=0}^{\overline n-1} \pr(N \ge n) x^n+\sum_{n=\overline n}^\infty \pr(N \ge n) x^n}{c 
\sum_{n=1}^{\overline n-1} x^n/n+ c \sum_{n=\overline n}^\infty  x^n/n} \\
& \sim 
\frac{\sum_{n=\overline n}^\infty \pr(N \ge n) x^n}{c \sum_{n=\overline n}^\infty  x^n/n} \in (1-\varepsilon,1+\varepsilon).
\end{split}
\]
This implies immediately that 
$n(1-G_N(\pr(R<n)))\sim n \, c\, \ln(1/\pr(R \ge n))\pr(R \ge n)$ and, again, Proposition~\ref{pro:homogeneous1}(1) yields the
result.
\item It follows analogously from Proposition~\ref{pro:homogeneous1}(2).
 \end{enumerate}
\end{proof}

\begin{proof}[Proof of Theorem~\ref{thm:3}]
We study the behavior of the deterministic counterpart.
Note that $\pr(\widetilde R_{i}<n-i+1)=G_{N_{i}}(\pr(R_{i} <n-i+1))$ is the probability that the
$n+1$-vertex does not belong to the radius of influence of the $i$-th vertex.
Hence $\prod_{i=0}^n G_{N_{i}}(\pr(R_{i} <n- i+1))$ is the probability
that the $n+1$-vertex does not belong to the radius of influence of any vertex to its left.
Denote this event by $E_n$:
by Borel-Cantelli $\pr(\limsup_n E_n)=1$. Whence, there exists $n_0$ such that 
for all $\pr \left (\bigcap_{k \ge n_0} E_k^\complement \right )>0$. Since $\pr(V_{n_0})>0$,
where $V_{n_0}$ is the event ``all the stations at $0,1,\ldots,n_0$ are activated'', we have
(using the FKG inequality)
\[
 \pr(V) \ge \pr \left (\bigcap_{k \ge n_0} E_k^\complement \Big | V_{n_0} \right ) \pr(V_{n_0})= 
\pr \left ( \bigcap_{k \ge n_0} E_k^\complement \cap V_{n_0}
\right ) \ge \pr \left ( \bigcap_{k \ge n_0} E_k^\complement \right ) \pr(V_{n_0})>0.
\]
In particular if we have a deterministic environment, say $N_i:=m_i \in \N$ for all $i \in \N$, and
\begin{equation}\label{eq:sumprod4}
 \sum_{n=0}^\infty \prod_{i=0}^n \pr(R_{i} < n-i+1)^{m_{i}}<+\infty
\end{equation}
then, since $G_{N_{i}}(x)=x^{m_{i}}$, equation~\eqref{eq:sumprod3} holds and $\pr(V)>0$.

Finally, by Lemma~\ref{lem:1} (using $t_i:= \pr(R_{i} < n-i+1)$) we see that equation~\eqref{eq:sumprod4}
holds for $\mu$-almost all configurations $\mathbf{n}$ and this yields the result.
  \end{proof}

\begin{proof}[Proof of Theorem~\ref{th:reversehomogeneous}]
 Apply \cite[Theorem 2.8]{cf:JMZ} to the deterministic annealed process.
Trivially $\E[\widetilde R]<+\infty$ if and only if $W<+\infty$.
 The
results follow immediately from the equivalence between
$\pr(S)=1$ (resp.~$\pr(S)=0$) and
$\pr(S|\mathcal{N}=\mathbf{n})=1$ (resp.~$\pr(S|\mathcal{N}=\mathbf{n})=0$) for $\mu$-almost all configurations $\mathbf{n}$.
\end{proof}

%
%
%

In the following, by $f \gtrsim g$ as $x \to x_0$ we mean that $\liminf_{x \to x_0} f(x)/g(x) \ge 1$.

\begin{proof}[Proof of Theorem~\ref{th:big}]
 \begin{enumerate}
  \item
We start by noting that, given any $\N$-valued random variable $N$ such that $0<\E[N]<+\infty$ and an
arbitrary sequence $\{t_n\}_{n \in \N}$ in $[0,1]$, then $\sum_{n \in \N} (1-G_N(t_n))<+\infty$
if and only if
$\sum_{n \in \N} (1-t_n)<+\infty$. This follows easily from (a) $\lim_{n \to \infty} (1-t_n) = 0$ $\Longleftrightarrow$ $\lim_{n \to \infty} (1-G_N(t_n)) = 0$ and
(b) $1-G_N(x) \sim \E[N](1-x)$ as $x \to 1^-$.
The result follows by taking $t_n:=\pr(R<n)$.
\item
Note that $W<+\infty$ if and only if $\E[\widetilde R]<+\infty$ where $\widetilde R$ is the law of the radius 
of the deterministic annealed counterpart of the reverse firework process.
Moreover
\[
 \begin{split}
  \int_0^{+\infty} (1-G_N(\pr(R < t)) \diff t &= \int_0^{+\infty} \sum_{n \in \N} (1-\pr(R < t)^n) \pr({N}=n) \diff t\\
&= \int_0^{+\infty} \sum_{n=1}^\infty (1-\pr(R < t)) \sum_{i=0}^{n-1} \pr(R < t)^i \pr({N}=n) \diff t\\
&=
\int_0^{+\infty} \pr(R \ge t) \sum_{i=0}^\infty \pr(R < t)^i \pr(N\ge i+1) \diff t\\
& =\int_0^{+\infty} f_{N,R}(t) \pr(R \ge t)  \diff t
 \end{split}
\]
where $f_{N,R}(t):=\sum_{i=0}^\infty \pr(R < t)^i \pr(N\ge i+1)$.
By the Monotone Convergence Theorem $f_{N,R}(t) \uparrow \E[N]$ as $t \uparrow +\infty$.

Since $\{\pr(N \ge n)\}_{n \in \N}$ is a nonincreasing sequence, then
from 
\cite[Sec.~XIII.5, Theorem 5]{cf:FII},
if $\pr(N \ge n) \ge \varepsilon L(n)/n^\alpha$ we have that
\[
 f_{N,R}(t) \ge \varepsilon \sum_{i=0}^\infty \pr(R < t)^i L(i+1)/(i+1)^\alpha \sim \varepsilon
\frac{\Gamma(1-\alpha)}{(1-\pr(R < t))^{1-\alpha}} L(1/(1-\pr(R < t))
\]
as $t \to +\infty$. This, in turn, implies
\[
 f_{N,R}(t)  \pr(R \ge t)  \gtrsim \varepsilon \Gamma(1-\alpha) \pr(R \ge t)^\alpha L(1/\pr(R \ge t))
\]
(remember that $\pr(R<t) \to 1$ as $t \to \infty$).
Hence $\int_0^{+\infty} \pr(R \ge t)^\alpha L(1/\pr(R \ge t)) \diff t =+\infty$ implies $W=+\infty$.

\item Analogously
\[
 f_{N,R}(t)  \pr(R \ge t)  \lesssim M \cdot \Gamma(1-\alpha) \pr(R \ge t)^\alpha L(1/\pr(R \ge t)),
\]
Hence $\int_0^{+\infty} \pr(R \ge t)^\alpha L(1/\pr(R \ge t)) \diff t <+\infty$ implies $W<+\infty$.
%

\item
In this case we have $\pr(N \ge n+1) \le  M/(n+1)$ for all $n \in \N$, 
whence
\[
 f_{N,R}(t)  \pr(R  \ge t)  \lesssim 
M \cdot  \pr(R \ge t) \ln(1/\pr(R \ge t)).
\]
Thus, $\int_0^{+\infty} \pr(R \ge t) \ln(1/\pr(R \ge t)) \diff t <+\infty$ implies $W<+\infty$.

\item
Finally $\pr(N \ge n+1) \ge  \varepsilon /(n+1)$ for all $n \in \N$, 
whence
\[
 f_{N,R}(t)  \pr(R \ge t)  \gtrsim \varepsilon \cdot  \pr(R \ge t) \ln(1/\pr(R \ge t)).
\]
Thus, $\int_0^{+\infty} \pr(R \ge t) \ln(1/\pr(R \ge t)) \diff t =+\infty$ implies $W=+\infty$.

 \end{enumerate}

\end{proof}

\begin{proof}[Proof of Theorem~\ref{th:reverseinhomogeneous}]
 If we define
\[
 \xi_{n}:= \sum_{k \ge 1} (1-\pr(R_{n+k}<k)^{N_{n+k}}), \qquad
\zeta:= \sum_{n \in \N} \prod_{k=1}^\infty \pr(R_{n+k}<k)^{N_{n+k}}
\]
then, by the Monotone Convergence Theorem and the  Bounded Convergence Theorem,
\[
 \E_\mu[\xi_{n}]=\sum_{k \ge 1} (1-G_{N_{n+k}}(\pr(R_{n+k}<k))),
\qquad
\E_\mu[\zeta]=\sum_{n \in \N} \prod_{k=1}^\infty G_{N_{n+k}}(\pr(R_{n+k}<k)).
\]
According to \cite[Theorem 2.4(i)]{cf:JMZ}, $\E_\mu[\xi_{n}]=+\infty$ if and only if $\pr(S)=1$ 
(almost sure survival of the deterministic annealed counterpart) which is
equivalent to $\pr(S|\mathcal{N}=\mathbf{n})=1$ for $\mu$-almost all configurations $\mathbf{n}$. Moreover,
$\E_\mu[\zeta]<+\infty$ implies $\zeta <+\infty$ for $\mu$-almost all configurations  $\mathbf{n}$ and then,
according to \cite[Theorem 2.4(ii)]{cf:JMZ}, $\pr(S|\mathcal{N}=\mathbf{n})>0$ for $\mu$-almost all configurations $\mathbf{n}$.
\end{proof}


Before proving the results of Section~\ref{sec:GWtree} we discuss again the role of the annealed counterpart
introduced in Section~\ref{subsec:annealed}. By using the equality $\pr(\widetilde R < t)=\ident_{(0,+\infty)}(t) 
G_N(\pr(R<t))$ we have that if the original process satisfies the assumptions in one of the results of
of Section~\ref{sec:GWtree} (Theorems~\ref{th:GW-Homogeneous-Firework2} and \ref{th:GW-Homogeneous-reverse-Firework2}
or Corollaries~\ref{cor:GW-Homogeneous-Firework} and \ref{cor:GW-Homogeneous-reverse-Firework2}) then the annealed counterpart
satisfies the same conditions by taking $N=1$ a.s.~and $\widetilde R$ instead of $R$
(observe that, in this case, $G_N(t)\equiv t$). This means that proving these results
in the special case $N=1$ a.s.~is equivalent to proving them for every annealed counterpart of a generic process.

\begin{proof}[Proof of Lemma~\ref{lem:BRWihenrited2}]
 $(1)$ and $(2)$ are consequences of equation~\eqref{eq:disintegration}
(in the second case one applies equation~\eqref{eq:disintegration} to $\pr(\cdot|\tau \textrm{ is infinite})$).

Let us prove $(3)$ and $(4)$.
 Given 
$\{R_{w,i}\}_{w \in \mathcal{W}, i \in \N^*}$, 
denote by $q(\Upsilon)=\pr(\textrm{extinction}|\tau^L=\Upsilon)$ the probability of extinction of the firework (resp.~reverse firework) process on a fixed labelled tree
$\Upsilon \in \mathbb{LT}$ (observe that, when conditioning on $\{\tau=\Upsilon\}$,
the radii are still random variables).
Remember that $\mu$ (defined in Section~\ref{subsec:genericlabelledtrees}) is the law of 
$\tau^L$ depending on
 $N_\emptyset$ (the law of the random label $l$ of the root). By hypothesis, the support
of the law of $N_\emptyset$ is a subset of $J_N$ (the support of the law of $N$).

Denote by $O$ the set $\{\Upsilon\colon q(\Upsilon)=1\}$ of labelled trees where the process dies out almost surely.
Clearly 
if $A:=\{\Upsilon \colon  q(\Upsilon)=1, l(\Upsilon) \ge 1\}$ 
then  $O=A \cup \mathbb{LT}_0$ (where $l(\Upsilon)$ is the label of the root of $\Upsilon$
and $\mathbb{LT}_0$ are the trees with no stations at the root). 
Note that $O^\complement=\{q(\Upsilon)<1\}\subseteq\{\Upsilon \textrm{ is infinite}, l(\Upsilon)  \ge 1\}$.

On a finite labelled tree $\Upsilon$ the firework and the reverse firework processes become extinct almost surely,
hence $E \subseteq O$. 
Using the notation 
$\mu_j(\cdot)=\mu(\cdot|l =j)$ (see Section~\ref{subsec:genericlabelledtrees}),
we have $\mu_j(E \setminus A)=\mu_j(A \triangle O)=0$ for all $j \in \N^*$.
If a process
becomes extinct on a labelled tree $\Upsilon \in A$ (since there is a positive probability that it reaches each child of the root) then
it becomes extinct on every labelled subtree branching from a child of the root.
%
%
%
%
Since each subtree can be identified with a labelled tree and the sequence of radii is i.i.d.,
we have that $(A,O)$ is inherited with respect to $\overline J=\N^*$. 
To be precise, for the above identification, 
in the case 
of the reverse firework, we just need that $\{R_{w,i}\}_{w \in \mathcal{W}\setminus\{\emptyset\}, i \in \N^*}$ is
an i.i.d.~family.
If the annealed probability of survival is positive
then by equation~\eqref{eq:disintegration} we have that 
$\mu(O)\equiv\pr(q(\tau^L)=1)<1$ which implies  $\pr(q(\tau^L)=1|N_\emptyset \ge 1)<1$,
where $q(\tau^L)=\pr(\textrm{extinction}|\tau^L)$.
Note that $\{q(\tau^L)=1\}=O^\complement \times \mathcal{O}$.

(3) Suppose that $\pr(N =0)=0$. 
Thus, $J_N \subseteq \overline J$ and, applying Lemma~\ref{lem:BRWihenrited}, 
 we have that either
$\mu_j(O \triangle E)=0$ for all $j\in \N^*$ or $\mu_j(O) =1$ for all $j\in J_N$. Whence 
$\mu_j(O) \in \{\alpha, 1\}$ for all $j\in J_N$. In particular, 
recalling that $\mu(l=j)=\pr(N_\emptyset=j)$, if $\overline{\varsigma}(i):=\mu(l=i|l \ge 1)\equiv\pr(N_\emptyset =i| N_\emptyset \ge 1)$ then 
we have $\mu_{\overline{\varsigma}}(\cdot)=\sum_{j \in \N} \overline{\varsigma}(j) \mu_j(\cdot)=
\sum_{j \in J_N} \overline{\varsigma}(j) \mu_j(\cdot)=\mu(\cdot|l \ge 1)$.
From Lemma~\ref{lem:BRWihenrited} either
$\mu_{\overline{\varsigma}}(O \triangle E)=0$  or $\mu_{\overline{\varsigma}} (O)=1$. This
implies $\mu_{\overline{\varsigma}}(O)=\mu(O|l \ge 1) \in \{\alpha, 1\}$ 
(since $\alpha=\mu_j(E)=\mu(E)=\mu_{\overline{\varsigma}}(E)$ for all $j \in \N$).
Observe that $\mu(O)<1$ if and only if $\mu_{\overline{\varsigma}}(O)=\alpha$; in this case
$\mu(O \cap \{l \ge 1\})=\alpha \mu(l \ge 1)= \mu(E \cap \{l \ge 1\})$, since the label
of the root and the finiteness of the tree are independent by construction.
Hence, when  $\mu(O)<1$ we have 
$O \cap \{l \ge 1\}=E \cap \{l \ge 1\}$ except for a $\mu$-null set
(since $O \supseteq E$, that is, there is no survival on a finite tree).
 Thus
either
$\mu(O)=1$ or 
$O=\mathbb{LT}_0 \cup E$ except for a $\mu$-null set (in this case 
$\mu(O)=\pr( N_\emptyset =0)+\pr(N_\emptyset \ge 1)\alpha$).
Equivalently, either $\mu(O^\complement)=0$ or 
$O^\complement=\{\Upsilon \textrm{ is infinite}, l(\Upsilon)  \ge 1\}$ except for a $\mu$-null set.
%
Equivalently,
$\{q(\tau^L)<1\}=\{\tau^L \textrm{ is infinite}, N_\emptyset  \ge 1\}$ except for a $\pr$-null set. 
This means that there is a
positive probability of survival for the process for $\mu$-almost every realization of the environment (that is, the labelled GW-tree) such that
the underlying tree is infinite and there is at least one station at the root. 
Easy computations shows that $\pr(\tau^L \textrm{ is infinite}, N_\emptyset  \ge 1)=\pr(N_\emptyset \ge 1)(1-\alpha)$.
This completes the proof of $(3)$.

(4) If $\pr(N=0)>0$ then consider the annealed process with one station per site and
radii $\{\widetilde R_{w}\}_{w \in \mathcal{W}}$ (see Section~\ref{subsec:annealed}).  
The environment of this process can be identified with the unlabelled tree 
$\tau$. Define $q_2(\mathbf{T})=\pr(\textrm{extinction}|\tau=\mathbf{T})$ the probability of extinction 
of the annealed process on $\mathbf{T}\in \mathbb{T}$.
By reasoning as in Section~\ref{subsec:annealed}, for every fixed realization $\mathbf{T}$ of $\tau$, 
$q_2(\mathbf{T})$ is the same for the annealed or the original process. 
Since the annealed process has one station per site, we apply (3) obtaining that 
$\pr(\textrm{survival})>0$ implies $\{q_2(\tau)<1\}=\{\tau \textrm{ is infinite}\}$
except for a $\pr$-null set. 

Denote now by $q_1(\mathbf{T},n)=\pr(\textrm{extinction}|\tau=\mathbf{T}, N_\emptyset=n)$ 
the probability of extinction of the firework (resp.~reverse firework) process on a fixed unlabelled tree
$\mathbf{T} \in \mathbb{T}$ with $n \in \N$ stations at the root
(observe that, when conditioning on $\{\tau=\mathbf{T}, N_\emptyset=n\}$
the radii of all stations and the numbers of stations outside the root are still random variables).
Clearly, in the firework process, the probability of survival starting from $n$ stations at the root is less than or equal to  $n$ times the probability of survival starting from $1$ station (since
it is necessary and sufficient that at least one of the $n$ stations triggers a surviving process).
Hence $1-q_1(\tau, 1) \le 1- q_1(\tau, n) \le n (1-q_1(\tau,1))$ for all $n \ge 1$.
For the reverse firework process the first inequality turns into an equality, since the behavior
of the process does not depend on the number of stations at the root as long as they are positive.
Since  $1-q_2(\tau)=\sum_{n \in \N^*}(1- q_1(\tau,n)) \pr_{N_\emptyset}(n)$ then, using the previous inequalities, 
$q_2(\tau)<1 \Longleftrightarrow q_1(\tau,n)<1 \textrm{ for some } n \in \mathrm{supp}(N_\emptyset) \Longleftrightarrow
q_1(\tau,n)<1, \  \forall n \in \mathrm{supp}(N_\emptyset)$. This implies that $\{q_1(\tau,N_\emptyset)<1, N_\emptyset=n\}=
\{\tau \textrm{ is infinite}, N_\emptyset=n\}$ for all $n \ge 1$ except for a $\pr$-null set.
This means that $\pr(\mathrm{survival}|\tau=\mathbf{T}, N_\emptyset =n)>0$ for $\overline \mu$-almost every infinite (unlabelled) tree 
$\mathbf{T}\in \mathbb{T}$
and for $\pr_{N_\emptyset}$-almost all $n \ge 1$. 
Observe that $q_1(\tau, N_\emptyset)=\pr(\textrm{extinction}|\tau, N_\emptyset)$; again $\pr(\tau \textrm{ is infinite}, N_\emptyset  \ge 1)=\pr(N_\emptyset \ge 1)(1-\alpha)$.
%
\end{proof}

\begin{proof}[Proof of Theorem~\ref{th:GW-Homogeneous-Firework2}]
The proof is divided into two main parts. We 
start in part (a) by proving the results in the case $N=1$ a.s.: the general case will be 
considered in (b).

(a). 
Here the environment $\tau^L$ can be identified with a realization of the GW-tree $\tau$
(with no labels) and the generating function of $N$ is $G_N(t) \equiv t$. 
As in Section~\ref{sec:intro} we denote by $\overline{\mu}$ the law of $\tau$ on $\mathbb{T}$.

(1-2). 
Consider at each step and for each station at a vertex $w$ the border of the signal $\partial_w$,
that is, the vertices at distance
$n$ from the source such that $n \le R_w< n+1$. Define a new process where at each step we activate only the stations
in $\partial_w$ where $w$ is the site of a station activated in the previous step. This process is stochastically dominated
by the original Firework process. 
The generating function of the law of $\# \partial_w$ can be computed as 
$\Psi(z):=\sum_{n \in \N} \pr(n \le R <n+1) \varphi^{(n)}_{\rho}(z)$ where $\varphi^{(n)}_{\rho}$
is
the generating function of the convolution $\underbrace{\rho * \ldots * \rho}_{n \textrm{ times}}$ 
(and $\varphi^{(0)}_{\rho}(z) \equiv 1$), that is, of the law of the number of descendants in the $n$-th generation.
%
%
The number of activated stations at each step is a GW-process with generating function $\Psi$ hence it survives
%
%
if and only if $\Psi^\prime(1)>1$ which is $\sum_{n=1}^\infty \pr(n \le R < n+1) m^n >1$. This inequality
is clearly equivalent to $\Phi(m)-1>\Phi(0)$.
Since
the (annealed) probability of survival is positive then $\pr(q(\tau)=1)<1$ hence, 
according to Lemma~\ref{lem:BRWihenrited2}(3),
$\{q(\tau)=1\}=
\{\tau \textrm{ is finite}\}$ except for a $\pr$-null set. 

(3). 
By using a coupling argument it is enough to prove the result for the process on the deterministic
tree $\mathbb{T}_k$ where each vertex has $k$ children. 
Consider the following random tree constructed on the vertices of $\mathbb{T}_k$.
We connect by a blue edge the root $\emptyset$ with each vertex $w \in \mathbb{T}_k$ such that the length $|w| \le R_\emptyset$.
This is the first generation. Suppose we have defined the $n$-th generation $\mathcal{H}_n$.
Take the vertices of $\mathcal{H}_n$, one at the time, in lexicographic order and connect  $w \in \mathcal{H}_n$ 
to every vertex $ww^\prime$ in the subtree of $\mathbb{T}_k$
branching from $w$ which has not been previously connected and such that the length $|w^\prime| \le R_w$.
The connected vertices $ww^\prime$ are the $n+1$-th generation.
The vertices of this random tree are exactly the vertices of $\mathbb{T}_k$ which are activated in the
firework process. More precisely, this is a spanning tree of the (random) F-graph described in 
Section~\ref{sec:intro}.
Hence, by coupling,
the firework process survives on $\mathbb{T}_k$ if and only if the blue tree is infinite.

But the blue random tree is a subtree of a new GW-tree whose (random) number of offsprings has the same law as
the number of descendants of any vertex of $\mathbb{T}_k$, say the root for instance, up to the random generation $R$ (included). 
%
%
%
%
It is easy to check that the
expected number of descendants up to the (random) $R$-th generation is  
$\sum_{n=1}^\infty \pr(n \le R < n+1) (k^{n+1}-k)/(k-1)$; thus if
this expected value is less or equal than $1$ the blue tree is finite almost surely and
the probability of survival of the firework process is $0$. 
The extinction for almost all the realizations of the GW-tree follows easily from equation~\eqref{eq:disintegration}.
The proof is complete in the case where $N=1$ almost surely.

(b). We consider now the general case. 
Let us denote by $\tau^L$ the labelled GW-tree and by $\tau$ the projection of $\tau^L$ on $\mathcal{W}$, that is,
the underlying (unlabelled) GW-tree. Let  $\mu$ be the law of $\tau^L$ on $\mathbb{LT}$. Given $\{R_{w,i}\}_{w \in \mathcal{W}, i \ge 1}$ and $\{N_w\}_{w \in \mathcal{W}}$,
we consider, along with the firework process $\eta$ on the labelled GW-tree, the annealed counterpart, that is,
the firework process $\widetilde \eta$ with one
station per vertex and radii
$\widetilde R_w :=\max(R_{w,1}, \ldots, R_{w,N_w})$ for all $w \in \mathcal{W}$. 
If $\eta$  satisfies the conditions of the theorem then $\widetilde \eta$
is a process with one station per site satisfying again the condition of the theorem (in the case studied in the previous part of the proof).

(1). Since $\pr(\widetilde R_w <t)=G_N(\pr(R<t))$ 
according to (1-2) we have that
if $\Phi(m)-1 >\Phi(0)$ (or equivalently
$\sum_{n=1}^\infty (G_N(\pr(R<n+1))-G_N(\pr(R<n))) m^n >1$) then the process
$\widetilde \eta$ survives with positive probability on almost every infinite GW-tree, that is,
$\overline{\mu} \big (\mathbf{T}\colon \pr(\widetilde S|\tau=\mathbf{T})>0 \big )=\overline{\mu}
\big (\mathbf{T}\colon \mathbf{T} \textrm{ is infinite}\big )$
(where $\widetilde S$ is the event ``$\widetilde \eta$ survives'' and $\overline{\mu}$ is the GW-probability measure
on the space $\mathbb{T}$ of unlabelled trees). 
This implies that $\pr(\widetilde S)>0$. Since $\pr(S)=\pr(\widetilde S)$ (where $S$ is the event ``$\eta$ survives'') then
we have $\mu \big (\Upsilon\colon \pr(S|\tau^L=\Upsilon)>0 \big )>0$. 
Since 
$\mu \big (\Upsilon\colon \pr(S|\tau^L=\Upsilon)>0 \big )>0$ then, according to 
Lemma~\ref{lem:BRWihenrited2}
(remember
that the support of the law $N_\emptyset$ is equal to $J_N$),
since $\pr(N =0)=0$, we have
$\{\pr(S|\tau^L)>0\}=\{\tau^L \not \in E, l(\tau^L)>0\}$ except for a $\pr$-null measure set
(where $E$ is the set of all finite labelled trees).

(2). As in (1) of part (b) we have that $\pr(S)=\pr(\widetilde S)>0$ which implies, according to 
Lemma~\ref{lem:BRWihenrited2},  
$\{\pr(S|\tau,N_\emptyset)>0\}=\{\tau \textrm{ is infinite}, N_\emptyset  \ge 1\}$
except for a $\pr$-null set.

(3). If 
$\Phi(k)-1 \le 1-1/k$ 
then $\overline \mu \big (\mathbf{T}\colon \pr(\widetilde S|\tau=\mathbf{T})=0 \big )=1$
which implies $0=\pr(\widetilde S)=\pr(S)$. This, in turn, is equivalent to 
$\mu \big (\Upsilon\colon \pr(S|\tau^L=\Upsilon)=0 \big )=1$.
\end{proof}

\begin{proof}[Proof of Corollary~\ref{cor:GW-Homogeneous-Firework}]
 Since $\pr(R<1) <1$ then $\Phi(m) > \Phi(0)+1$ eventually as $m \to \infty$; thus, according to 
Theorem~\ref{th:GW-Homogeneous-Firework2}, $\overline m_c<+\infty$.
Since the GW-tree is a.s.~finite when $m \le 1$, we have that $\overline m_c \ge 1$.
Moreover, if $m > 1/ \limsup_{n \to \infty} \sqrt[n]{1-G_N(\pr(R<n))}$ 
(the latter being the radius of convergence of $\Phi$ according to the discussion before Corollary~\ref{cor:GW-Homogeneous-Firework}) 
then, by Theorem~\ref{th:GW-Homogeneous-Firework2},  there is survival with positive probability for the firework process.
This implies that if $\limsup_{n \to \infty} \sqrt[n]{1-G_N(\pr(R<n))}=1$ we have positive survival if and only if $m >1$. 
Thus $
\overline m_c=1$.

\end{proof}

%
%

\begin{proof}[Proof of Theorem~\ref{th:GW-Homogeneous-reverse-Firework2}]
As in the proof of Theorem~\ref{th:GW-Homogeneous-Firework2} we start with the case 	where $N=1$ almost surely.
In the first part of the proof, the environment $\tau^L$ is identified with the GW-tree $\tau$
(since we have one station per vertex).

After the construction of the GW-tree $\tau$, 
we construct a new tree  (which is not in general a subtree of $\tau$)
iteratively as follows:
starting from the origin, (1) for every word $i$ of length 1 of the tree
draw a purple edge from $\emptyset$ to $i$ if and only if $R_{i} \ge1$,
(2) for all $n \ge 1$ suppose we finished connecting words of length $n$, take all
words $w^\prime$ of length $n+1$ such that there are no ancestors already connected to the root
and connect them to the root if and only if $R_{w^\prime} \ge n+1$. This is the construction of the first purple generation. 
Now we construct the second generation applying (1) and (2) to the subtrees branching from
each vertex of the 1st purple generation (in the construction, these vertices become the roots of 
the branching subtrees). The constructions on these subtrees are independent since 
the subtrees are pairwise disjoint.
The construction of the subsequent generations follows by iteration.
Using this construction, if a vertex is able to listen to more than one station which can broadcast the signal 
then we are connecting it
to the closest one. This is a spanning tree of the (random) RF-graph described in Section~\ref{sec:intro}. 

This new purple tree is a GW-tree and there is survival
%
%
%
if and only if this tree is infinite. The expected number
of purple edges from the origin is 
$\phi_2(m)=\sum_{i=1}^\infty m^i \pr(R \ge i) \prod_{j=1}^{i-1} \pr(R<j)
$ where $m$ is the expected number of children
in the original GW-tree. 
Hence, the purple GW-tree is finite if and only if the probability that a vertex has at least one
child is strictly smaller than 1 and the expected number of children of a vertex is smaller or equal than 1.

Given the original GW-process $\{Z_n\}_{n \in \N}$, 
the probability (conditioned on  $\{Z_n\}_{n \in \N}$)
that the root has no children in the purple process is $\prod_{n=1}^\infty \pr(R<n)^{Z_n}$. Clearly
$\prod_{n=1}^\infty \pr(R<n)^{Z_n}=0$ if and only if $\sum_{n=1}^\infty {Z_n}\pr(R\ge n)=+\infty$. 
For almost every realization of  $\{Z_n\}_{n \in \N}$ such that the GW-tree is infinite, we have $Z_n \sim m^n$ (where $m>1$),
thus  $\sum_{n=1}^\infty {Z_n}\pr(R\ge n)=+\infty$  is equivalent to
$\phi_1(m)=\sum_{n=1}^\infty {m^n}\pr(R\ge n)=+\infty$.
%
%

Hence there is annealed a.s.~extinction if and only if $\phi_1(m)<+\infty$ and $\phi_2(m) \le 1$;
equation~\eqref{eq:disintegration} implies a.s.~extinction on almost every realization of the GW-tree. Moreover
the annealed probability of survival conditioned of the event `` the GW-tree is infinite'' 
is $1$ if and only if $\phi_1(m)=+\infty$; again 
equation~\eqref{eq:disintegration} implies survival with probability 1 on almost every realization of the GW-tree.
%
When $\phi_1(m)<+\infty$ and $\phi_2(m) >1$ we apply 
Lemma~\ref{lem:BRWihenrited2} to obtain the quenched results.
As in the proof of Theorem~\ref{th:GW-Homogeneous-Firework2}, 
the results in the general case come from the first part of the proof and from Lemma~\ref{lem:BRWihenrited2} 
by using
the reverse firework process $\eta$ (associated with $\{R_{w,i}\}_{w \in \mathcal{W}, i \ge 1}$ and $\{N_w\}_{w \in \mathcal{W}}$) 
and its annealed counterpart $\widetilde \eta$ 
(with one station per vertex and radii
$\widetilde R_w :=\max(R_{w,1}, \ldots, R_{w,N_w})$ for all $w \in \mathcal{W}$).
\end{proof}

\begin{proof}[Proof of Corollary~\ref{cor:GW-Homogeneous-reverse-Firework2}]
Recall $\phi_1$ and $\phi_2$ defined in   Theorem~\ref{th:GW-Homogeneous-reverse-Firework2}
From the probabilistic interpretation of $\phi_1(m)$ and $\phi_2(m)$ given in the proof
of Theorem~\ref{th:GW-Homogeneous-reverse-Firework2} 
we have that $\phi_2(m) <1$ implies $\phi_1(m)<+\infty$. Moreover, $\phi_1 \ge \phi_2$.
Since we assumed that $\pr(R<1) \in (0, 1)$ we have $G_N(\pr(R<1)) \in (0, 1)$ as well. In particular, 
$G_N(\pr(R<1))<1$, implies that 
$\phi_2$ is strictly increasing and $\lim_{m \to \infty} \phi_2(m)=+\infty$. 
Define
\[
 M_c:= \sup\{m \ge 0 \colon \phi_1(m) <+\infty\}, \quad  m_c:= \sup \{m \ge 0 \colon \phi_2(m) \le 1 \} \equiv 
\sup \{m \ge 0 \colon \phi_2(m) < 1 \}.
\]
By the discussion above, $m_c \le M_c$ and $m_c < +\infty$. 
Observe that, in general,
$\sum_{n=1}^\infty (1-\alpha_n)\prod_{j=1}^{n-1} \alpha_j = 1 -\lim_{n \to \infty} \prod_{j=1}^{n} \alpha_j$ when the limit exists.
Hence $\phi_2(1)=
1-\prod_{j=1}^\infty G_N(\pr(R<j))$ which implies 
$m_c \ge 1$
(note that $\phi_2(1)<1$ if and only if $\phi_1(1)<+\infty$).
 
Since $\phi_2$ is a series with nonnegative coefficients, we have that $\phi_2(m_c) \le 1$; thus if $m_c < M_c$ then
$\phi_1(m_c)<+\infty$ and for $m=m_c$ 
there is almost sure extinction for almost every realization of the environment.
\end{proof}

\begin{proof}[Details on Remark~\ref{rem:phasetransition2}]
Recall the definition of $\phi_1$ and $\phi_2$ given in the proof of Corollary~\ref{cor:GW-Homogeneous-reverse-Firework2}.
If $M_c=+\infty$ there is nothing to prove since, according to Corollary~\ref{cor:GW-Homogeneous-reverse-Firework2},
$m_c<+\infty$. On the other hand,
 suppose that $1<M_c <+\infty$. 
This implies immediately
 that $\phi_1(1)< +\infty$ hence $\prod_{j=1}^\infty \pr(R<j) =\delta >0$.
Thus, $\phi_1(m) \ge \phi_2(m) \ge \delta 
\phi_1(m)$. 
Since $\phi_2$ is a series with nonnegative coefficients, we have that $\phi_2(m_c) \le 1$,
whence $\phi_1(m_c) \le 1/\delta<+\infty$. If, in addition, $\phi_1(M_c)= +\infty$ then
we have that $M_c > m_c$.
\end{proof}


\end{document}